\documentclass[english,leqno,10pt,a4paper]{article}
\usepackage{amsmath,amstext, amsthm, amssymb}

\usepackage{pifont}
\usepackage[hypertex]{hyperref}
\usepackage{srcltx}
\usepackage[totalheight=24 true cm, totalwidth=17 true cm]{geometry}
\usepackage[english]{babel}
\usepackage[all]{xy}
\usepackage{color}
\newtheorem{thmintro}{Theorem}

\newtheorem{conjintro}[thmintro]{Conjecture}

\newtheorem{thm}{Theorem}[section]
\newtheorem{cor}[thm]{Corollary}
\newtheorem{lemma}[thm]{Lemma}
\newtheorem{prop}[thm]{Proposition}

\theoremstyle{remark}
\theoremstyle{definition}
\newtheorem{defn}[thm]{Definition}
\newtheorem{parag}[thm]{}
\newtheorem{rmk}[thm]{Remark}

\numberwithin{equation}{section}

\def\beq{\begin{equation}}
\def\eeq{\end{equation}}

\def\crash#1{}

\def\N{{\mathbb N}}
\def\Z{{\mathbb Z}}
\def\Q{{\mathbb Q}}
\def\R{{\mathbb R}}
\def\C{{\mathbb C}}

\def\P{{\mathbb P}}
\def\F{{\mathbb F}}

\def\l{\left}
\def\r{\right}
\def\[[{\l[\l[}
\def\]]{\r]\r]}
\def\p{\prime}

\def\sgq{\sigma_q}

\def\Sgq{\Sigma_q}
\def\dq{d_q}

\def\ord{{\rm ord}}

\def\cf{\emph{cf. }}
\def\ie{\emph{i.e. }}
\def\ds{\displaystyle}

\def\cA{{\mathcal A}}
\def\cB{{\mathcal B}}
\def\cC{{\mathcal C}}

\def\cF{{\mathcal F}}
\def\cM{{\mathcal M}}
\def\cN{{\mathcal N}}
\def\cH{{\mathcal H}}

\def\cL{{\mathcal{ L}}}
\def\cO{{\mathcal O}}
\def\cP{{\mathcal P}}

\def\cW{{\mathcal W}}

\def\wtilde{\widetilde}

\def\ul{\underline}
\def\ol{\overline}

\def\a{\alpha}
\def\be{\beta}
\def\de{\delta}
\def\ga{\gamma}
\def\sg{\sigma}

\def\la{\lambda}
\def\La{\Lambda}
\def\Na{\nabla}
\def\Sg{\Sigma}
\def\De{\Delta}

\def\Dq{\Delta_q}

\def\deg{\mathop{\rm deg}}

\def\GL{\mathop{\rm GL}}

\author{Lucia Di Vizio
\thanks{Lucia DI VIZIO,
Laboratoire de Mathématiques UMR 8100, CNRS,
Universit\'e de Versailles-St Quentin,
45 avenue des États-Unis
78035 Versailles cedex, France.
{\tt divizio@math.cnrs.fr}}
and
Charlotte Hardouin
\thanks{Charlotte HARDOUIN, Institut de Math\'{e}matiques de Toulouse,
118 route de Narbonne,
31062 Toulouse Cedex 9, France.
{\tt hardouin@math.univ-toulouse.fr}}}

\date{\today}

\title%[Generic Galois groups for $q$-difference equations]
{On the Grothendieck conjecture on $p$-curvatures
for $q$-difference equations
\thanks{Work partially supported by ANR-06-JCJC-0028.}}

\begin{document}
%\makeatletter\c@page=267\makeatother
\bibliographystyle{alpha}

\maketitle
%%%%%%%%%%%%%%%%%%%%%%%%%%%%%%%%%%%%%%%%%%%%%%%%%%%%%%%%%%%%%%%%%%%%
%%%%%%%%%%%%%%%%%%%%%%%%%%%%%%%%%%%%%%%%%%%%%%%%%%%%%%%%%%%%%%%%%%%%
%%%%%%%%%%%%%%%%%%%%%%%%%%%%%%%%%%%%%%%%%%%%%%%%%%%%%%%%%%%%%%%%%%%%
%%%%%%%%%%%%%%%%%%%%%%%%%%%%%%%%%%%%%%%%%%%%%%%%%%%%%%%%%%%%%%%%%%%%
%%%%%%%%%%%%%%%%%%%%%%%%%%%%%%%%%%%%%%%%%%%%%%%%%%%%%%%%%%%%%%%%%%%%
%%%%%%%%%%%%%%%%%%%%%%%%%%%%%%%%%%%%%%%%%%%%%%%%%%%%%%%%%%%%%%%%%%%%

\begin{abstract}
In the present paper, we give a $q$-analogue of
Grothendieck conjecture on $p$-curvatures for $q$-difference equations
defined over the field of rational function $K(x)$, where $K$ is a finite
extension of a field of rational functions $k(q)$, with $k$ perfect.
Then we consider the generic (also called intrinsic) Galois group in the
sense of \cite{Katzbull} and \cite{DVInv}.
The result in the first part of the paper lead to a description of the generic Galois group
through the properties of the functional equations obtained specializing $q$
on roots of unity.
Although no general Galois correspondence holds in
this setting, in the case of positive characteristic, where nonreduced groups appear, we can prove some devissage
of the generic Galois group.
\par
In the last part of the paper,
we give a complete answer to the analogue of
Grothendieck conjecture on $p$-curvatures for $q$-difference equations
defined over the field of rational function $K(x)$, where $K$ is any finitely generated
extension of $\Q$ and $q\neq 0,1$: we prove that the generic Galois group of a $q$-difference module over $K(x)$
always admits an adelic description in the spirit of the Grothendieck-Katz conjecture.
To this purpose, if $q$ is an algebraic number, we prove a generalization of the results in \cite{DVInv}.
\end{abstract}

\maketitle

\setcounter{tocdepth}{1}
\tableofcontents

%%%%%%%%%%%%%%%%%%%%%%%%%%%%%%%%%%%%%%%%%%%%%%%%%%%%%%%%%%%%%%%%%%%%
%%%%%%%%%%%%%%%%%%%%%%%%%%%%%%%%%%%%%%%%%%%%%%%%%%%%%%%%%%%%%%%%%%%%
%%%%%%%%%%%%%%%%%%%%%%%%%%%%%%%%%%%%%%%%%%%%%%%%%%%%%%%%%%%%%%%%%%%%
\section*{Introduction}
%\addcontentsline{toc}{section}{Introduction}
%%%%%%%%%%%%%%%%%%%%%%%%%%%%%%%%%%%%%%%%%%%%%%%%%%%%%%%%%%%%%%%%%%%%
%%%%%%%%%%%%%%%%%%%%%%%%%%%%%%%%%%%%%%%%%%%%%%%%%%%%%%%%%%%%%%%%%%%%
%%%%%%%%%%%%%%%%%%%%%%%%%%%%%%%%%%%%%%%%%%%%%%%%%%%%%%%%%%%%%%%%%%%%

In the present paper, we give a complete answer to the Grothendieck-Katz's conjecture for $q$-difference equation
proving the following two open cases.
First of all, we allow $q$ to be a transcendental parameter.
In \cite{DVInv}, there was no hope of recovering information on the
classical Grothendieck conjecture for differential equations by letting $q$ tends to $1$, for lack of an appropriate topology.
On the contrary, here we allow $q$ to be a parameter and therefore we can recover a differential equation specializing $q$ to $1$.
Secondly, we generalize the results in \cite{DVInv}, proved for a number field, to the case of a finitely generated extension of $\Q$.
Thanks to those results, we show that the Galois group of a linear $q$-difference equation with coefficients in a field of rational functions
$K(x)$, with $K$ field of characteristic zero, can always be characterized in the spirit of the Grothendieck-Katz conjecture, via some curvatures.
In this way we give an ``adelic'' answer to the direct problem for Galois theory of $q$-difference equations, for which we do not have any kind of algorithm.

\medskip
The question of the algebraicity of solutions of differential or difference equations goes back at
least to Schwarz, who established in 1872 an exhaustive
list of hypergeometric differential equations having a full set of algebraic solutions.
Galois theory of linear differential equations, and more recently Galois theory of linear difference equations, have been developed to investigate the existence of algebraic
relations between the solutions of linear functional equations via the computation of a linear algebraic group called the Galois group.
In particular, the existence of a basis of algebraic solutions is essentially equivalent to having a finite Galois group.
The computation of these Galois groups thus provides a powerful tool to study the algebraicity of special functions.
The direct problem in differential Galois theory (\ie for differential equations) was solved by Hrushovski in \cite{Hrushcomp}.
Although he actually has a computational algorithm, the calculations of the Galois group of a differential equation is
still a very difficult problem, most of the time, out of reach.
For difference Galois theory, the existence of a general computational algorithm is still an open question.
\par
Grothendieck-Katz conjecture on $p$-curvatures conjugates these two aspects of the theory:
determining whether a differential equation has a full basis of
algebraic solutions
and solving the direct problem.
Thanks to
Grothendieck's conjecture on $p$-curvatures we have a (necessary and) sufficient
conjectural condition to test whether the solutions of a differential equation are algebraic or not.
More precisely, one can reduce a differential equation
$$
{\mathcal L}y=a_\mu(x){d^\mu y\over dx^\mu}+
a_{\mu-1}(x){d^{\mu-1} y\over dx^{\mu-1}}+
\dots+a_0(x)y=0,
$$
with coefficients in the field $\Q(x)$, modulo $p$ for almost all primes $p\in\Z$.
Then Grothendieck's conjecture, which remains open in full generality
(\cf \cite{Andregrothendieck})
predicts:

\begin{conjintro}[Grothendieck's conjecture on $p$-curvatures]
The equation ${\mathcal L}y=0$ has a full set of algebraic solutions
if (and only if)
for almost all primes $p\in\Z$ the reduction modulo
$p$ of ${\mathcal L}y=0$ has a full set of solutions in ${\F}_p(x)$.
\end{conjintro}

After \cite{Katzbull}, this is equivalent to the fact that
\emph{the Lie algebra of the generic Galois group of a differential module
$\cM=(M,\Na)$ over $\Q(x)$ is the smallest algebraic Lie subalgebra of
${\rm End}_{\Q(x)}(M)$ whose reduction modulo $p$ contains the $p$-curvature
for almost all $p$}. We are not explaining here the precise meaning of this statement,
to whom a large literature is devoted.

\medskip
The first and main result of the paper is the following.
We consider a perfect field $k$, the field of rational functions $K=k(q)$ and
a linear $q$-difference equation
$$
\cL y(x):=a_\nu(q,x)y(q^\nu x)+\dots+a_1(q,x)y(qx)+a_0(q,x)y(x)=0
$$
with coefficients in $K(x)$.
Note that below we will take $K$ to be a finite extension of $k(q)$, which we avoid to do here not to introduce
too technical notation.
\par
Then:

\begin{thmintro}\label{thmintro:grothnaive}
The equation $\cL y(x)=0$ has a full set of solutions in $K(x)$, linearly independent
over $K$, if and only if for almost all positive integer $\ell$ and all
primitive roots of unity $\zeta_\ell$ of order $\ell$
the equation
$$
a_\nu(\zeta_\ell,x)y(\zeta_\ell^\nu x)+\dots+a_1(\zeta_\ell,x)y(\zeta_\ell x)+a_0(\zeta_\ell,x)y(x)=0
$$
has a full set of
solutions in $k(\zeta_\ell)(x)$, linearly independent over $k(\zeta_\ell)$.
\end{thmintro}

We will denote by $\sgq$ the $q$-difference operator $f(x)\mapsto f(qx)$,
acting on any algebra where it make sense to consider it (for instance $K(x)$, $K((x))$ etc.).
A $q$-difference module $\cM_{K(x)}=(M_{K(x)},\Sgq)$ over $K(x)$ is a $K(x)$-vector space of finite dimension $\nu$
equipped with a $\sgq$-semilinear bijective operator $\Sgq$:
$$
\Sgq(fm)=\sgq(f)\Sgq(m),\,
\hbox{for any $m\in M$ and $f\in K(x)$.}
$$
One goes from a linear $q$-difference equation to a linear $q$-difference system and, then, to
a linear $q$-difference module, and
viceversa, in the same way as one goes from linear differential equations to linear differential
systems and to vector bundles with connection.
\par
One can always find
a polynomial $P(x)\in k[q,x]$ and
a $q$-difference algebra of the form
$$
\cA=k\l[q,x,\frac{1}{P(x)},\frac{1}{P(qx)},\frac{1}{P(q^2x)},...\r],
$$
so that there exists an $\cA$-lattice $M$ of $M_{K(x)}$, stable by $\Sgq$.
We obtain in this way a $q$-difference module $\cM=(M,\Sgq)$ over $\cA$, that allows to recover
$\cM_{K(x)}$ by extension of scalars.
We denote by $\{\phi_v\}_v$ the family of all irreducible polynomials in $k[q]$ whose roots are root of unity.
For any $v$ we denote by $\kappa_v$ the order of the roots of $\phi_v$.
We prove the following statement, which is equivalent to Theorem \ref{thmintro:grothnaive} above:

\begin{thmintro}[\cf Theorem \ref{thm:GrothKaz} below]\label{thmintro:GrothKaz}
The $q$-difference module $\cM_{K(x)}$ is trivial if and only if
the operator $\Sgq^{\kappa_v}$ acts as the identity on
$M\otimes_{k[q]}\frac{k[q]}{(\phi_v)}$ for almost all $v$.
\end{thmintro}

The module $\cM_{K(x)}$ is said to be trivial if it is isomorphic
to $(K^\nu\otimes_K K(x),1\otimes\sgq)$, for some positive integer $\nu$ or, equivalently if
an associated linear $q$-difference equation in a cyclic basis has a full set of rational solutions.
Notice that $\Sgq^{\kappa_v}$ induces
an $\frac{k[q]}{(\phi_v)}$-linear map on $M\otimes_{k[q]}\frac{k[q]}{(\phi_v)}$.
\par
We consider the collection $Constr(\cM_{K(x)})$
of $K(x)$-linear algebraic constructions of $\cM_{K(x)}$ (direct sums, tensor product, symmetric and antisymmetric product, dual).
The operator $\Sgq$ induces a $q$-difference operator on every
element of $Constr(\cM_{K(x)})$, that we will still call $\Sgq$.
Then, the generic Galois group of $\cM_{K(x)}$ is defined as:
$$
\begin{array}{l}
Gal(\cM_{K(x)},\eta_{K(x)})=
\{\varphi\in \GL(M_{K(x)}):
\hbox{$\varphi$ stabilizes}\\
\hskip 40 pt\hbox{every subset stabilized by $\Sgq$, in any construction}\}
\end{array}
$$
As in \cite{Katzbull}, Theorem \ref{thmintro:GrothKaz} is equivalent to the following
statement, whose precise meaning is explained in
\S\ref{sec:genericgaloisgroup}:

\begin{thmintro}[\cf Theorem \ref{thm:genGalois} below]\label{thmintro:genGalois}
The generic Galois group $Gal(\cM_{K(x)},\eta_{K(x)})$ is the smallest algebraic subgroup
of $\GL(M_{K(x)})$ that contains the operators $\Sgq^{\kappa_v}$ modulo $\phi_v$ for almost all $v$.
\end{thmintro}

In the case of positive characteristic, the group $Gal(\cM_{K(x)},\eta_{K(x)})$
is not necessarily reduced. Although
there is no Galois correspondence for generic Galois groups,
in the nonreduced case we can prove some devissage.
In fact, let $p>0$ be the characteristic of $k$ and let
us consider the short exact sequence associated with the largest
reduced subgroup $Gal_{red}(\cM_{K(x)},\eta_{K(x)})$ of
the generic Galois group $Gal(\cM_{K(x)},\eta_{K(x)})$:
$$
1\longrightarrow Gal_{red}(\cM_{K(x)},\eta_{K(x)})
\longrightarrow Gal(\cM_{K(x)},\eta_{K(x)})
\longrightarrow \mu_{p^\ell}\longrightarrow 1.
$$
Then we have (\cf Theorem \ref{thm:genGaloisred} and Corollary \ref{cor:diffgenGaloisred} below):

\begin{thmintro}
\begin{itemize}
\item
The group $Gal_{red}(\cM_{K(x)},\eta_{K(x)})$ is the smallest algebraic subgroup
of $\GL(M_{K(x)})$ that contains the operators $\Sgq^{\kappa_v p^\ell}$ modulo $\phi_v$ for almost all $v\in\cC$.

\item
$Gal_{red}(\cM_{K(x)},\eta_{K(x)})$ is the generic Galois group of the $q^{p^\ell}$-difference
module
$(M_{K(x)},\Sgq^{p^\ell})$.

\item
Let $\wtilde K$ be a finite extension of $K$ containing a $p^\ell$-th root $q^{1/p^\ell}$ of $q$.
The generic Galois group $Gal(\cM_{\wtilde K(x^{1/p^\ell})},\eta_{\wtilde K(x^{1/p^\ell})})$ is reduced
and
$$
Gal(\cM_{\wtilde K(x^{1/p^\ell})},\eta_{\wtilde K(x^{1/p^\ell})})
\subset Gal_{red}(\cM_{K(x)},\eta_{K(x)})
\otimes_{K(x)}\wtilde K(x^{1/p^\ell}).
$$
\end{itemize}
\end{thmintro}

In the last section of the paper we apply the previous results to the characterization of
the generic Galois group of a $q$-difference module over $\C(x)$, with $q\in\C\smallsetminus\{0,1\}$.
We prove a statement that can be summarized informally in the
following way:

\begin{thmintro}\label{thmintro:katzcomplex}
The generic Galois group of a complex $q$-difference module $\cM=(M,\Sgq)$
is the smallest algebraic subgroup of $\GL(M)$, that contains a cofinite nonempty subset of
curvatures.
\end{thmintro}

This means that there exists a field $K$, finitely generated over $\Q$ and containing $q$,
and a $q$-difference module $\cM_{K(x)}$ over $K(x)$,
such that
$$
\l\{\begin{array}{l}
\cM=\cM_{K(x)}\otimes\C(x),\\
Gal(\cM,\eta_{\C(x)})=Gal(\cM_{K(x)},\eta_{K(x)})\otimes \C(x).
\end{array}\r.
$$
Then:
\begin{itemize}
\item
If $q$ is a transcendental number, we are in the situation studied in the present paper.

\item
If $q$ is a root of unity, the theorem is proved for algebraic generic Galois groups in \cite{Hendrikstesi}.

\item
If $q$ is an algebraic number and $K$ is a number field, we are in the situation studied in \cite{DVInv}.
\end{itemize}
In \S\ref{sec:modulescarzero} below, Theorem \ref{thmintro:katzcomplex}
is proved in the only remaining open case, namely under the assumption that
$q$ is algebraic and $K/\Q$ is finitely generated, but not finite.
Notice that the arguments of \cite{DVInv} use crucially the fact that number fields satisfy the product formula,
therefore they cannot be applied to this situation.
Here we consider a basis of transcendence of $K/\Q$ and consider its elements as parameters to be specialized in the field of algebraic
numbers. After the specialization we apply the results in \cite{DVInv} and then we recover the information over $K(x)$.
To prove this last step, we uses Sauloy's canonic solutions and Birkhoff connection matrix as constructed in \cite{Sfourier} .

\medskip
The results of this paper rise the following considerations:
\begin{trivlist}

\item
\emph{Comparison theorem with other differential Galois theories (see \cite{diviziohardouinComp}).}
There are many different Galois theories for $q$-difference equations.
In \cite{diviziohardouinComp}, we elucidate the comparison of all of them with the generic Galois group introduced above.
We point out that Theorem \ref{thmintro:katzcomplex} is an important ingredient in the proof of the comparison theorems.
We prove in particular that the dimension of the generic Galois group over $K(x)$ is equal to the transcendency
degree of the extension generated over the field of rational functions with elliptic coefficients over
$C^*/q^\Z$ by a full set of solutions of the $q$-difference equation.

\item
\emph{Specialization of the parameter $q$ (see \cite{diviziohardouinComp}).}
We have proved that, when $q$ is a parameter, \ie when it is transcendental over $k$,
independently of the characteristic, the structure of a $q$-difference equation
is totally determined by the structure of the $\xi$-difference equations obtained specializing
$q$ to almost all primitive root of unity $\xi$.
In \cite{diviziohardouinComp}, we prove that the specialization of the algebraic generic Galois group
at $q=a$ for any $a$ in the algebraic closure of $k$, contains the algebraic generic
Galois group of the specialized equation. In other words, we can say that the information spreads from the roots of unity to
the other possible specializations of $q$.

\item
\emph{Link to the classical Grothendieck conjecture (see \cite{diviziohardouinComp}).}
If $k$ is a number field, we can reduce the equations to positive characteristic, so that
$q$ reduces to a parameter to positive characteristic.
So, if we have a $q$-difference equation $Y(qx)=A(q,x)Y(x)$ with coefficients in a field $k(q,x)$ such that $[k:\Q]<\infty$,
we can either reduce it to positive characteristic and then specialize $q$, or specialize $q$ and then reduce
to positive characteristic. In particular, letting $q\to 1$ in
$$
\frac{Y(qx)-Y(x)}{(q-1)x}=\frac{A(q,x)-1}{(q-1)x}Y(x)
$$
we obtain a differential system.
In \cite{diviziohardouinComp},
we elucidate some of the inclusions of the Galois groups that we find in this way.
This could give some new
method to tackle the Grothendieck conjecture for differential equations, by confluence
and q-deformation of a differential equation. The idea would be to find
a suitable q-deformation to translate the arithmetic of the curvatures of the linear
differential equation into the q-arithmetic of the curvatures of the q-difference
equation obtained by deformation.

\item
\emph{Link with the Galois theory of parameterized difference equations (see \cite{diviziohardouinqMalg}).}
In \cite{HardouinSinger}, the authors attach to a $q$-difference equation
a parameterized Galois group, which is a differential algebraic group \emph{\`a la Kolchin}.
This is a subgroup of the group of invertible matrices defined by a set of algebraic differential equations.
The differential dimension of this Galois group measures
the hypertranscendence properties of a basis of solutions. We recall that
a function $f$ is hypertranscendental over a field $F$ equipped with a derivation $\partial$
if $F[\partial^n(f), n\geq 0]/F$ is a transcendental extension of infinite degree, or equivalently,
if $f$ is not a solution of a nonlinear algebraic differential equation with coefficients in
$F$.
\par
In \cite{diviziohardouinqMalg}, we combine the Grothendieck conjecture with the differential approach to difference equations
of Hardouin-Singer to obtain a parameterized generic Galois group, arithmetically characterized in the spirit
of Grothendieck-Katz conjecture.
This allows us
to establish the relation between Malgrange-Granier $D$-groupoid for nonlinear $q$-difference equations and
the Galois theory of linear $q$-difference equations, comparing the first one to the parameterized
generic Galois group.
Doing this we build the first path between
Kolchin's theory of linear differential algebraic groups and Malgrange's D-groupoid,
answering a question of B. Malgrange (see [Mal09, page 2]).
It should lead to
an arithmetic approach to integrability in the spirit of
\cite{casaleroques}. Notice that the problem of differential dependency with respect to
the derivation $\frac{d}{dq}$, when $q$ is transcendental, is quite peculiar and is treated in \cite{diviziohardouinPacific}.
\end{trivlist}

\subsubsection*{Acknowledgements.}
We would like to thank D. Bertrand, Z. Djadli, C. Favre, M. Florence, A. Granier, D. Harari, F. Heiderich, A. Ovchinnikov, B. Malgrange, J-P. Ramis, J. Sauloy, M. Singer, J. Tapia, H. Umemura and M. Vaquie for the many discussions on different points of this paper, and the
organizers of the seminars of the universities of Grenoble I, Montpellier, Rennes II, Caen, Toulouse and Bordeaux
that invited us to present the results below at various stages of our work.
\par
We would like to thank the ANR projet Diophante that has made possible a few reciprocal visits,
and the Centre International de Rencontres
Math\'{e}matiques in Luminy for supporting us \emph{via} the Research in pairs program and
for providing a nice working atmosphere.

%%%%%%%%%%%%%%%%%%%%%%%%%%%%%%%%%%%%%%%%%%%%%%%%%%%%%%%%%%%%%%%%%%%%
%%%%%%%%%%%%%%%%%%%%%%%%%%%%%%%%%%%%%%%%%%%%%%%%%%%%%%%%%%%%%%%%%%%%
%%%%%%%%%%%%%%%%%%%%%%%%%%%%%%%%%%%%%%%%%%%%%%%%%%%%%%%%%%%%%%%%%%%%
\section{Notation and definitions}
\label{sec:not}
%%%%%%%%%%%%%%%%%%%%%%%%%%%%%%%%%%%%%%%%%%%%%%%%%%%%%%%%%%%%%%%%%%%%
%%%%%%%%%%%%%%%%%%%%%%%%%%%%%%%%%%%%%%%%%%%%%%%%%%%%%%%%%%%%%%%%%%%%
%%%%%%%%%%%%%%%%%%%%%%%%%%%%%%%%%%%%%%%%%%%%%%%%%%%%%%%%%%%%%%%%%%%%

\begin{parag}{\bfseries The base field.}
Let us consider the field of rational function $k(q)$ with coefficients
in a perfect field $k$.
We fix $d\in]0,1[$ and for any irreducible polynomial $v=v(q)\in k[q]$
we set:
$$
|f(q)|_v=d^{\deg_q v(q)\cdot\ord_{v(q)}f(q)},\,\forall f(q)\in k[q].
$$
The definition of $|~|_v$ extends to $k(q)$ by multiplicativity.
To this set of norms one has to add the $q^{-1}$-adic one, defined on $k[q]$ by:
$$
|f(q)|_{q^{-1}}=d^{-deg_qf(q)}.
$$
Once again, this definition extends by multiplicativity to $k(q)$. Then, the
product formula holds:
$$
\begin{array}{rcl}
\prod_{v\in k[q]\,{\rm irred.}}\l|\frac{f(q)}{g(q)}\r|_v
&=&d^{\sum_v\deg_q v(q)~\l(\ord_{v(q)}f(q)-\ord_{v(q)}g(q)\r)}\\
&=&d^{deg_q f(q)-\deg_q g(q)}\\
&=&\l|\frac{f(q)}{g(q)}\r|_{q^{-1}}^{-1}.\\
\end{array}
$$
For any finite extension $K$ of $k(q)$, we consider the family
$\cP$ of ultrametric norms, that extends the norms defined above, up to equivalence.
We suppose that the norms in $\cP$ are normalized so that the product formula still holds.
We consider the following partition of $\cP$:
\begin{itemize}
\item
the set $\cP_\infty$ of places of $K$ such that the associated
norms extend, up to equivalence, either $|~|_q$ or $|~|_{q^{-1}}$;

\item

the set $\cP_f$ of places of $K$ such that the associated
norms extend, up to equivalence, one of the norms $|~|_v$
for an irreducible $v=v(q)\in k[q]$, $v(q)\neq q$.\footnote{The notation $\cP_f$, $\cP_\infty$ is only psychological,
since all the norms involved here are ultrametric. Nevertheless, there exists
a fundamental difference between the two sets, in fact for any $v\in\cP_\infty$ one has
$|q|_v\neq 1$, while for any $v\in\cP_f$ the $v$-adic norm of $q$ is $1$. Therefore,
from a $v$-adic analytic point of view, a $q$-difference equation has a totally
different nature with respect to the norms in the sets $\cP_f$ or $\cP_\infty$.}
\end{itemize}

Moreover we consider the set $\cC$ of places $v\in\cP_f$
such that $v$ divides a valuation of $k(q)$ having as uniformizer
a factor of a cyclotomic polynomial, other than $q-1$.
Equivalently, $\cC$ is the set
of places $v\in\cP_f$ such that $q$ reduces to a root of unity modulo $v$ of order strictly greater
than $1$.
We will call $v\in\cC$ a cyclotomic place.
\par
Sometimes we will write $\cP_K$, $\cP_{K,f}$, $\cP_{K,\infty}$ and $\cC_K$, to
stress out the choice of the base field.
\end{parag}

\begin{parag}{\bfseries $q$-difference modules.}\label{parag:defqmod}
\label{subsec:submodules}
The field $K(x)$ is naturally a $q$-difference algebra, \ie is equipped with the
operator
$$
\begin{array}{rccc}
\sgq:&K(x)&\longrightarrow&K(x)\\
&f(x)&\longmapsto&f(qx)
\end{array}.
$$
The field $K(x)$ is also equipped with the $q$-derivation
$$
\dq(f)(x)=\frac{f(qx)-f(x)}{(q-1)x},
$$
satisfying a $q$-Leibniz formula:
$$
\dq(fg)(x)=f(qx)\dq(g)(x)+\dq(f)(x)g(x),
$$
for any $f,g\in K(x)$.
\par
More generally, we will consider a field $K$, with a fixed element $q\neq 0$, and
an extension $\cF$ of $K(x)$ equipped with a $q$-difference operator, still called $\sgq$, extending
the action of $\sigma_q$, and with the skew derivation $\dq:=\frac{\sgq-1}{(q-1)x}$.
Typically, in the sequel, we will consider the fields $K(x)$, $K(x^{1/r})$, $r\in\Z_{>1}$, or
$K((x))$.
\par
A $q$-difference module over $\cF$ (of rank $\nu$)
is a finite dimensional $\cF$-vector space $M_{\cF}$ (of dimension $\nu$)
equipped with an invertible $\sgq$-semilinear operator, \ie
$$
\Sgq(fm)=\sgq(f)\Sgq(m),
\hbox{~for any $f\in \cF$ and $m\in M_{\cF}$.}
$$
A morphism of $q$-difference modules over $\cF$ is a morphism
of $\cF$-vector spaces, commuting with the $q$-difference structures
(for more generalities on the topic, \cf \cite{vdPutSingerDifference},
\cite[Part I]{DVInv}
or \cite{gazette}). We denote by $Diff(\cF, \sgq)$ the category of $q$-difference modules over $\cF$.
\par
Let $\cM_{\cF}=(M_{\cF},\Sgq)$ be a $q$-difference module over $\cF$ of rank $\nu$.
We fix a basis $\ul e$ of $M_{\cF}$ over $\cF$ and we set:
$$
\Sgq\ul e=\ul e A,
%\,\,\ds \Dq^n\ul e=\ul e G_n(x).
$$
with $A\in \GL_\nu(\cF)$.
A horizontal vector $\vec y\in \cF^\nu$ with respect to the basis $\ul e$ for
the operator $\Sgq$ is a vector
that verifies $\Sgq(\ul e\vec y)=\ul e\vec y$, \ie $\vec y=A\sgq(\vec y)$. Therefore we call
$$
\sgq(Y)=A_1Y,
\hbox{~with~}A_1=A^{-1},
$$
the ($q$-difference) system associated to $\cM_{\cF}$ with respect to the basis $\ul e$.
Recursively, we obtain a family of higher order $q$-difference systems:
$$
\sgq^n(Y)=A_n Y
\hbox{~and~}\dq^nY=G_nY,
$$
with $A_n \in \GL_\nu(\cF)$
and $G_n \in M_\nu(\cF)$. Notice that:
$$
A_{n+1}= \sgq(A_n)A_1,\,
G_1=\frac{A_1-1}{(q-1)x}
\hbox{~and~}
G_{n+1}=\sgq(G_n)G_1(x)+\dq G_n.
$$
It is convenient to set $A_0=G_0=1$.
Moreover we set $[n]_q=\frac{q^n-1}{q-1}$,
$[n]_q^!=[n]_q[n-1]_q\cdots[1]_q$, $[0]_q^!=1$ and $G_{[n]}=\frac{G_n}{[n]_q^!}$ for any $n\geq 1$.
\end{parag}

\begin{parag}{\bfseries Reduction modulo places of $K$.}\label{parag:redmodK}
In the sequel, we will deal with an arithmetic situation, in the following sense.
We consider the ring of integers $\cO_K$ of $K$, \ie the integral closure of $k[q]$ in $K$, and a $q$-difference algebra
of the form
\beq\label{eq:algebraA}
\cA=\cO_K\l[x,\frac1{P(x)},\frac1{P(qx)},\frac1{P(q^2x)},...\r],
\eeq
for some $P(x)\in\cO_K[x]$.
Then $\cA$ is stable by the action of $\sgq$ and we can consider a free
$\cA$-module $M$ equipped with a semilinear invertible
operator\footnote{We could have asked
that $\Sgq$ is only injective, but then, enlarging the scalar to a
$q$-difference algebra $\cA^\p\subset K(x)$, of the same form as \eqref{eq:algebraA},
we would have obtained an invertible operator. Since we are interested in the reduction of
$\cM$ modulo almost all places of $K$, we can suppose
without loss of generality that $\Sgq$ is invertible.} $\Sgq$.
Notice that $\cM_{K(x)}=(M_{K(x)}=M\otimes_\cA K(x),\Sgq\otimes\sgq)$
is a $q$-difference module over
$K(x)$. We will call $\cM=(M,\Sgq)$ a $q$-difference module over $\cA$.
\par
Any $q$-difference module over $K(x)$ comes from
a $q$-difference module over $\cA$, for a convenient choice of $\cA$.
The reason for considering $q$-difference modules over $\cA$ rather than over $K(x)$,
is that we want to reduce our $q$-difference modules with respect to
the places of $K$, and, in particular, with respect to the cyclotomic
places of $K$.
\par
We denote by $k_v$ the residue field of $K$
with respect to a place $v\in\cP$, $\pi_v$ the uniformizer of $v$ and $q_v$ the
image of $q$ in $k_v$, which is defined for all places $v\in\cP$.
For almost all $v\in\cP_f$ we can consider
the $k_v(x)$-vector space $M_{k_v(x)}=M\otimes_\cA k_v(x)$, with the structure induced
by $\Sgq$.
In this way, for almost all $v\in\cP$, we obtain a $q_v$-difference module
$\cM_{k_v(x)}=(M_{k_v(x)},\Sg_{q_v})$ over $k_v(x)$,
\par
In particular, for almost all $v\in\cC$, we obtain a $q_v$-difference module
$\cM_{k_v(x)}=(M_{k_v(x)},\Sg_{q_v})$ over $k_v(x)$, having the particularity that $q_v$ is
a root of unity, say of order $\kappa_v$. This means that $\sg_{q_v}^{\kappa_v}=1$ and that
$\Sg_{q_v}^{\kappa_v}$ is a $k_v(x)$-linear operator.
The results in \cite[\S2]{DVInv} apply to this situation. We recall some of them.
Since we have:
$$
\sg_{q_v}^{\kappa_v}=1+(q-1)^{\kappa_v} x^{\kappa_v}d_{q_v}^{\kappa_v}
\hskip 10 pt\hbox{and}\hskip 10 pt
\Sg_{q_v}^{\kappa_v}=1+(q-1)^{\kappa_v} x^{\kappa_v}\De_{q_v}^{\kappa_v},
$$
where $\De_{q_v}=\frac{\Sg_{q_v}-1}{(q_v-1)x}$, the following facts are equivalent:
\begin{enumerate}
\item
$\Sg_{q_v}^{\kappa_v}$ is the identity;
\item
$\De_{q_v}^{\kappa_v}$ is zero;
\item
the reduction of $A_{\kappa_v}$ modulo $\pi_v$ is the identity matrix;
\item
the reduction of $G_{\kappa_v}$ modulo $\pi_v$ is zero.
\end{enumerate}

\begin{defn}
If the conditions above are satisfied we say that $\cM$ has \emph{zero $\kappa_v$-curvature (modulo $\pi_v$)}.
We say that $\cM$ has \emph{nilpotent $\kappa_v$-curvature (modulo $\pi_v$)} or \emph{has nilpotent reduction}, if
$\Delta_{q_v}^{\kappa_v}$ is a nilpotent operator or equivalently if
$\Sg_{q_v}^{\kappa_v}$ is a unipotent operator.
\end{defn}

We will use this notion in \S\ref{sec:formalsolution}, while in \S\ref{sec:trivsol}
we will need the following stronger notion.
\end{parag}

\begin{parag}{\bfseries $\kappa_v$-curvatures (modulo $\phi_v$).}
\label{subsec:cyclocurvkappav}
We denote by $\phi_v$ the uniformizer of the cyclotomic place of $k(q)$ induced by $v\in\cC_K$. The ring $\cA\otimes_{\cO_K}\cO_K/(\phi_v)$
is not reduced in general, nevertheless it has a $q$-difference algebra structure
and the results in \cite[\S2]{DVInv} apply again. Therefore we set:

\begin{defn}
A $q$-difference module $\cM$ has \emph{zero $\kappa_v$-curvature (modulo $\phi_v$)} if the operator
$\Sgq^{\kappa_v}$ induces the identity
(or equivalently if
the operator $\De_q^{\kappa_v}$ induces the zero operator)
on the module
$M\otimes_\cA\cA/(\phi_v)$.
\end{defn}

\begin{rmk}
The rational function $\phi_v$ is, up to a multiplicative constant,
a factor of a cyclotomic polynomial for almost all $v$.
It is a divisor of $[\kappa_v]_q$ and $|\phi_v|_v=|[\kappa_v]_q|_v=|[\kappa_v]_q^!|_v$.
\end{rmk}

We recall the definition of the Gauss norm associated to an ultrametric norm $v\in\cP$:
$$
\hbox{for any~}\frac{\sum{a_i}x^i}{\sum{b_j}x^j}\in K(x),\hskip 10 pt
\l|\frac{\sum{a_i}x^i}{\sum{b_j}x^j}\r|_{v,Gauss}=\frac{\sup|a_i|_v}{\sup|b_j|_v}.
$$

\begin{prop}\label{prop:iteratedred}
Let $v\in\cC_K$.
We assume that $\l|G_1(x)\r|_{v,Gauss}\leq 1$.
Then the following assertions are equivalent:
\begin{enumerate}
\item
The module $\cM=(M,\Sgq)$ has zero $\kappa_v$-curvature modulo $\phi_v$.

\item
For any positive integer $n$, we have
$\l|G_{[n]}\r|_{v,Gauss}\leq 1$, \ie
the operator $\frac{\De_q^n}{[n]_q^!}$
induces a well defined operator on $\cM_{k_v(x)}=(M_{k_v(x)},\Sg_{q_v})$.

\end{enumerate}
\end{prop}

\begin{rmk}
Even if $q_v$ is a root of unity,
the family of operators $\frac{d_{q_v}^n}{[n]_{q_v}^!}$
acting on $k_v(x)$ is well defined.
This remark is the starting point
for the theory of iterated $q$-difference modules
constructed in \cite[\S3]{HardouinIterative}.
Then the second assertion of the proposition above can be rewritten as:
\begin{center}
\emph{$\cM_{k_v(x)}$ has a natural structure of iterated
$q_v$-difference module.}
\end{center}
\end{rmk}

\begin{proof}
The only nontrivial implication is ``$1\Rightarrow 2$''
whose proof is quite similar to \cite[Lemma 5.1.2]{DVInv}.
The Leibniz Formula for $\dq$ and $\Dq$ implies that:
$$
G_{(n+1)\kappa_v}=\sum_{i=0}^{\kappa_v}
{\kappa_v\choose i}_q \sgq^{\kappa_v-i} (\dq^i\l(G_{n\kappa_v}\r))
G_{\kappa_v-i},
$$
where ${n\choose i}_q=\frac{[n]_q^!}{[i]_q^![n-i]_q^!}$ for any $n\geq i\geq 0$.
If $\cM$ has zero $\kappa_v$-curvature modulo $\phi_v$ then
$|G_{\kappa_v}|_{v,Gauss}\leq |\phi_v|_v$.
One obtains recursively that
$|G_{m}|_{v,Gauss}\leq |\phi_v|_v^{\l[\frac{m}{\kappa_v}\r]}$,
where we have denoted by $[a]$ the integral part of $a\in\R$,
\ie $[a]=max\{n\in\Z:n\leq a\}$.
Since $|[\kappa_v]_q|_v=|\phi_v|_v$ and $|[m]_q^!|_v=|\phi_v|_v^{\l[\frac{m}{\kappa_v}\r]}$,
we conclude that:
\beq\label{eq:nilpinequality}
\l|\frac{G_m}{[m]_q^!}\r|_{v,Gauss}
\leq 1.
\eeq
\end{proof}
\end{parag}

%%%%%%%%%%%%%%%%%%%%%%%%%%%%%%%%%%%%%%%%%%%%%%%%%%%%%%%%%%%%%%%%%%%%
%%%%%%%%%%%%%%%%%%%%%%%%%%%%%%%%%%%%%%%%%%%%%%%%%%%%%%%%%%%%%%%%%%%%
%%%%%%%%%%%%%%%%%%%%%%%%%%%%%%%%%%%%%%%%%%%%%%%%%%%%%%%%%%%%%%%%%%%%
\section{Regularity of ``global'' nilpotent $q$-difference modules}
\label{sec:formalsolution}
%%%%%%%%%%%%%%%%%%%%%%%%%%%%%%%%%%%%%%%%%%%%%%%%%%%%%%%%%%%%%%%%%%%%
%%%%%%%%%%%%%%%%%%%%%%%%%%%%%%%%%%%%%%%%%%%%%%%%%%%%%%%%%%%%%%%%%%%%
%%%%%%%%%%%%%%%%%%%%%%%%%%%%%%%%%%%%%%%%%%%%%%%%%%%%%%%%%%%%%%%%%%%%

In this section, we are going to prove that a $q$-difference module
is regular singular and has integral exponents if it has nilpotent reduction
for sufficiently many cyclotomic places.
In this setting, and in particular if the characteristic of $k$ is zero,
speaking of global nilpotence is a little bit abusive. Nevertheless, it is
the terminology used in arithmetic differential equations and we think that
it is evocative of the ideas that have inspired what follows.

\begin{defn}\label{defn:regsing}
A $q$-difference module $(M,\Sgq)$ over $\cA$
(or another sub-$q$-difference algebra of $K((x))$)
is said to be regular singular at $0$ if there
exists a basis $\ul e$ of $(M\otimes_\cA K((x)),\Sgq\otimes\sgq)$ over $K((x))$
such that the action of $\Sgq\otimes\sgq$ over $\ul e$ is represented by a
constant matrix $A\in \GL_\nu(K)$.
\end{defn}

\begin{rmk}
It follows from the Frobenius algorithm\footnote{\cf \cite{vdPutSingerDifference} or \cite[\S1.1]{Sfourier}. The algorithm is briefly summarized
also in \cite[\S1.2.2]{SauloyENS} and \cite{gazette}.}, that a $q$-difference module $M_{K(x)}$
over $K(x)$ is regular singular if and only if
there exists a basis $\ul e$ such that $\Sgq\ul e=\ul e A(x)$ with $A(x)\in \GL_\nu(K(x))\cap \GL_n(K[[x]])$.
\end{rmk}

The eigenvalues
of $A(0)$ are called the exponents of $\cM$ at zero. They are well defined
modulo $q^\Z$.
The $q$-difference module $\cM$ is said to be
regular singular \emph{tout court} if it is regular singular
both at $0$ and at $\infty$, \ie after a variable change of the form $x=1/t$.

\par
In the notation of the previous section, we prove the following result,
which is actually an analogue of
\cite[\S13]{KatzTurrittin} (\cf also \cite{DVInv} for a $q$-difference version over a
number field):

\vbox{\begin{thm}\label{thm:FormalSol}~
\begin{enumerate}
\item
If a $q$-difference module $\cM$ over $\cA$ has nilpotent
$\kappa_v$-curvature modulo $\pi_v$ for infinitely many $v\in\cC$ then
it is regular singular.
\item
Let $\cM$ be a $q$-difference module over $\cA$.
If there exists an infinite set of positive primes $\wp\subset\Z$
such that $\cM$
has nilpotent
$\kappa_v$-curvature modulo $\pi_v$ for all $v\in\cC$
such that $\kappa_v\in\wp$,
then $\cM$ is regular singular and
its exponents (at zero and at $\infty$) are all in $q^\Z$.
\end{enumerate}
\end{thm}}

\begin{rmk}
The proof of the first part of Theorem \ref{thm:FormalSol}
is inspired by \cite[13.1]{KatzTurrittin} and therefore is quite similar
to \cite[\S6]{DVInv}. On the other hand, the proof of the triviality of the
exponents (\cf Proposition \ref{prop:ratexp} below) has significant differences
with respect to the analogous results on number fields.
In fact in the differential case the proof is based on
Chebotarev density theorem. In \cite{DVInv} it is a consequence on some considerations on
Kummer extensions and Chebotarev density theorem, while in this setting it is
a consequence of Lemma \ref{lemma:rationalfunctions}
below, which can be interpreted as a statement in rational dynamic.
\end{rmk}

The proof of Theorem \ref{thm:FormalSol}
is the object of the following three subsections.

\subsection{Regularity}
\label{subsec:regularity}

We prove the first part of Theorem \ref{thm:FormalSol}.
It is enough to prove that $0$ is a regular singular point for $\cM$,
the proof at $\infty$ being completely analogous.
\par
Let $r\in\N$ be a divisor of $\nu!$ where $\nu$ is the dimension of $M_{K(x)}$ over $K(x)$ and let $L$ be a finite extension of $K$ containing an
element $\wtilde q$
such that $\wtilde q^{\,r}=q$. We consider the field extension $K(x)\hookrightarrow L(t)$, $x\mapsto t^r$.
The field $L(t)$ has a natural structure of
$\wtilde q$-difference algebra extending the $q$-difference structure of $K(x)$.

\begin{lemma}\label{basis change}
The $q$-difference module $\cM$ is regular singular at $x=0$
if and only if
the $\wtilde q$-difference module $\cM_{L(t)}:=(M\otimes_{\cA} L(t),
\Sg_{\wtilde q}:=\Sgq\otimes\sg_{\wtilde q})$ over $L(t)$
is regular singular at $t=0$.
\end{lemma}

\begin{proof}
It is enough to notice that if $\ul e$ is a basis for $\cM$, then $\ul e\otimes 1$ is a basis
for $\cM_{L(t)}$ and $\Sg_{\wtilde q}(\ul e\otimes 1)=\Sgq(\ul e)\otimes 1$.
The other implication is a consequence of the Frobenius algorithm (\cf \cite{vdPutSingerDifference} or
\cite{Sfourier}).
\end{proof}

The next lemma can be deduced from the formal classification of $q$-difference modules
(\cf \cite[Corollary 9 and \S9, 3)]{Praag}, \cite[Theorem 3.1.7]{sauloyfiltration}):

\begin{lemma}\label{rationalgauge}
There exist an extension $L(t)/K(x)$ as above, a basis $\ul f$ of the
$\wtilde q$-difference module $\cM_{L(t)}$ and
an integer $\ell$
such that $\Sg_{\wtilde q}\ul f=\ul f B(t)$, with $B(t)\in \GL_\nu(L(t))$ of the following form:
\beq\label{nonnilp}
\left\{%
\begin{array}{l}
  \ds B(t)=\frac{B_\ell}{t^\ell}+\frac{B_{\ell-1}}{t^{\ell-1}}+\dots, \hbox{as an element of $\GL_\nu(L((t)))$;} \\
  \hbox{$B_\ell$ is a constant nonnilpotent matrix.} \\
\end{array}%
\right.
\eeq
\end{lemma}

\medskip
\begin{proof}[Proof of the first part of Theorem \ref{thm:FormalSol}]
Let $\cB\subset L(t)$ be a $\wtilde q$-difference algebra over the ring
of integers $\cO_L$ of $L$, of the same form
as \eqref{eq:algebraA}, containing the entries of $B(t)$.
Then there exists a
$\cB$-lattice $\cN$ of $\cM_{L(t)}$ inheriting the $\wtilde q$-difference
module structure from $\cM_{L(t)}$
and having the following properties:\\
1. $\cN$ has nilpotent reduction modulo infinitely many cyclotomic places of $L$;\\
2. there exists a basis $\ul f$ of $\cN$ over $\cA$ such that $\Sg_{\wtilde q}\ul f=\ul f B(t)$ and $B(t)$ verifies (\ref{nonnilp}).
\par
Iterating the operator $\Sg_{\wtilde q}$ we obtain:
$$
\Sg_{\wtilde q}^m(\ul f)
=\ul fB(t)B({\wtilde q}t)\cdots B({\wtilde q}^{m-1}t)
=\ul f\l(\frac{B_\ell^m}{\wtilde q^{\ell m(\ell m-1)\over 2}t^{m\ell}}+h.o.t.\r).
$$
We know that for infinitely many cyclotomic places $w$ of $L$,
the matrix $B(t)$ verifies
\beq\label{uniprel}
\l(B(t)B({\wtilde q}t)\cdots B({\wtilde q}^{\kappa_w-1}t)
-1\r)^{n(w)}\equiv 0\hbox{\ mod $\pi_w$},
\eeq
where $\pi_w$ is an uniformizer of the place $w$,
$\kappa_w$ is the order $\wtilde q$ modulo $\pi_w$
and $n(w)$ is a convenient positive integer.
Suppose that $\ell\neq 0$. Then $B_\ell^{\kappa_w}\equiv 0$ modulo $\pi_w$,
for infinitely many $w$, and hence
$B_\ell$ is a nilpotent matrix, in contradiction with Lemma \ref{rationalgauge}.
So necessarily $\ell=0$.
\par
Finally we have $\Sg_{\wtilde q}(\ul f)=\ul f\l(B_0+h.o.t\r)$. It follows from \eqref{uniprel}
that $B_0$ is actually invertible, which implies that $\cM_{L(t)}$ is regular singular at $0$.
Lemma \ref{basis change} allows to conclude.
\end{proof}

\subsection{Triviality of the exponents}

Let us prove the second part of Theorem \ref{thm:FormalSol}.
We have already proved that $0$ is a regular singularity
for $\cM$. This means that there exists a basis
$\ul e$ of $\cM$ over $K(x)$ such that $\Sgq\ul e=\ul eA(x)$, with
$A(x)\in \GL_\nu(K[[x]])\cap \GL_\nu(K(x))$.
\par
The Frobenius algorithm (\cf \cite[\S1.1.1]{Sfourier})
implies that there exists a shearing transformation $S\in \GL_\nu(K[x,1/x])$,
such that $S(qx)A(x)S(x)^{-1}\in \GL_\nu(K[[x]])\cap \GL_\nu(K(x))$ and that the constant term $A_0$
of $S(x)^{-1}A(x)S(qx)$ has the following properties:
if $\a$ and $\be$ are eigenvalues of $A_0$ and $\a\be^{-1}\in q^\Z$, then $\a=\be$.
So choosing the basis $\ul eS(x)$ instead of $\ul e$, we can assume that $A_0=A(0)$ has
this last property.
\par
Always following the Frobenius algorithm (\cf \cite[\S1.1.3]{Sfourier}),
one constructs recursively the
entries of a matrix $F(x)\in \GL_\nu(K[[x]]))$, with $F(0)=1$,
such that we have $F(x)^{-1}A(x)F(qx)=A_0$.
This means that there exists
a basis $\ul f$ of $\cM_{K((x))}$ such that $\Sgq\ul f=\ul fA_0$.
\par
The matrix $A_0$ can be written as the product of a semi-simple matrix $D_0$
and a unipotent matrix $N_0$.  Since $\cM$ has nilpotent reduction, we deduce from \S \ref{parag:redmodK} that the reduction of $A_{\kappa_v}=A_0^{\kappa_v}$ modulo $\pi_v$
is the identity matrix. Then $D_0$ verifies:
\beq\label{eq:nilpotence}
\hbox{for all $v\in\cC$ such that $\kappa_v\in\wp$,
we have $D_0^{\kappa_v}\equiv 1$ modulo $\pi_v$.}
\eeq
Let $\wtilde K$ be a finite extension of $K$
in which we can find all the eigenvalues of $D_0$. Then
any eigenvalue $\a\in\wtilde K$ of $A_0$ has the property that
$\a^{\kappa_v}=1$ modulo $w$, for all $w\in\cC_{\wtilde K}$, $w\vert v$ and $v$
satisfies \eqref{eq:nilpotence}.
In other words, the reduction modulo $w$ of an eigenvalue $\a$ of $A_0$
belongs to the multiplicative cyclic group generated by the reduction of $q$ modulo $\pi_v$.
\par
To end the proof, we have to prove that
$\a\in q^\Z$.
So we are reduced to prove the proposition below.

\begin{prop}\label{prop:ratexp}
Let $K/k(q)$ be a finite extension and
$\wp\subset\Z$ be an infinite set of positive primes. For any $v\in\cC$, let
$\kappa_v$ be the order of $q$ modulo $\pi_v$, as a root of unity.
\par
If $\a\in K$ is such that $\a^{\kappa_v}\equiv 1$ modulo $\pi_v$ for all
$v\in\cC$ such that $\kappa_v\in\wp$,
then $\a\in q^\Z$.
\end{prop}

\begin{rmk}
Let $K=\Q(\wtilde q)$, with $\wtilde q^r=q$, for some integer $r>1$.
If $\wtilde q$ is an eigenvalue of $A_0$ we would be asking that
for infinitely many
positive primes $\ell\in\Z$ there exists a primitive root of unity $\zeta_{r\ell}$ of order ${r\ell}$,
which is also a root of unity of order $\ell$. Of course, this cannot be true, unless
$r=1$.
\end{rmk}

\subsection{Proof of Proposition \ref{prop:ratexp}}
\label{subsec:ratexp}

We denote by $k_0$ either the field of rational numbers $\Q$, if the characteristic of $k$
is zero, or the field with
$p$ elements $\mathbb F_p$, if the characteristic of $k$ is $p>0$.
First of all, let us suppose that $k$ is a finite perfect extension of $k_0$ of degree $d$ and
fix an embedding $k\hookrightarrow\ol k$ of $k$ in its algebraic closure $\ol k$.
In the case of a rational function $f\in k(q)$, Proposition \ref{prop:ratexp}
is a consequence of the following lemma:

\begin{lemma}\label{lemma:rationalfunctions}
Let $[k:k_0]=d<\infty$ and let $f(q)\in k(q)$ be nonzero rational function.
If there exists an infinite set of positive primes $\wp\subset\Z$ with the following property:
\begin{quote}
for any $\ell\in\wp$ there exists a primitive root of unity $\zeta_\ell$ of order $\ell$ such that
$f(\zeta_\ell)$ is a root of unity of order $\ell$,
\end{quote}
then $f(q)\in q^\Z$.
\end{lemma}

\begin{rmk}\label{rmk:rationalfunctions}
If $k=\C$ and $y-f(q)$ is irreducible in $\C[q,y]$,
the result can be deduced from \cite[Chapter 8, Theorem 6.1]{LangRootsof1}, whose proof uses B\'ezout theorem.
We give here a totally elementary proof, that holds also in positive characteristic.
\par
Proposition \ref{prop:ratexp} can be rewritten in the language of rational dynamic.
We denote by $\mu_\ell$ the group of root of unity of order $\ell$.
The following assertions are equivalent:
\begin{enumerate}
\item
$f(q)\in k(q)$ satisfies the assumptions of Lemma \ref{lemma:rationalfunctions}.
\item
There exist infinitely many $\ell\in\N$ such that the group
$\mu_\ell$ of roots of unity of order $\ell$ verifies $f(\mu_\ell)\subset\mu_\ell$.
\item
$f(q)\in q^\Z$.
\item
The Julia set of $f$ is the unit circle.
\end{enumerate}
As it was pointed out to us by C. Favre, the equivalence between the last
two assumptions is a particular case of \cite{zdunik}, while the equivalence
between the second and the fourth assumption can be deduced from
\cite{FavreLetelier} or \cite{autissier}.
\end{rmk}

\begin{proof}
Let $f(q)=\frac{P(q)}{Q(q)}$, with $P=\sum_{i=0}^D a_i q^i,
Q=\sum_{i=0}^D b_i q^i\in k[q]$ coprime polynomials of degree less equal to $D$,
and let $\ell$ be a prime such that:
\begin{itemize}
\item $f(\zeta_\ell)\in\mu_\ell$;
\item $2D < \ell-1$.
\end{itemize}
Moreover, since $\wp$ is infinite, we can chose $\ell>>0$ so that
the extensions $k$ and $k_0(\mu_\ell)$ are linearly disjoint over $k_0$.
Since $k$ is perfect, this implies that the minimal polynomial of the primitive $\ell$-th
root of unity $\zeta_\ell$ over $k$ is $\chi(X)=1+X+...+X^{\ell-1}$.
Now let $\kappa\in\{0,\dots,\ell-1\}$ be such that $f(\zeta_\ell)=\zeta_\ell^\kappa$, \ie
$$
\sum_{i=0}^D a_i \zeta_\ell^i=\sum_{i=0}^D b_i \zeta_\ell^{i+\kappa}.
$$
We consider the polynomial $H(q)=\sum_{i=0}^D a_i q^i -\sum_{j=\kappa}^{D+\kappa}b_{j-\kappa}q^j$
and distinguish three cases:
\begin{enumerate}

\item
If $D+\kappa <\ell-1$, then $H(q)$ has $ \zeta_\ell$ as a zero
and has degree
strictly inferior to $\ell-1$. Necessarily $H(q)=0$. Thus we have
$$
a_0=a_1=...=a_{\kappa-1}=b_{D+1-\kappa}=...=b_D=0
\hskip 10pt\hbox{and}\hskip10pt
a_i=b_{i-\kappa}\hbox{~for~}i=\kappa,\dots,D,
$$
which implies $f(q)=q^\kappa$.

\item
If $D+\kappa=\ell-1$ then $H(q)$ is a $k$-multiple of $\chi(q)$
and therefore all the coefficients of $H(q)$ are all equal.
Notice that the inequality $D+\kappa\geq\ell-1$ forces
$\kappa$ to be strictly bigger than $D$, in fact
otherwise one would have $\kappa+D\leq 2D <\ell-1$.
For this reason the coefficients of $H(q)$ of the monomials
$q^{D+1},\dots,q^\kappa$ are all equal to zero.
Thus
$$
a_0=a_1=...=a_D=b_0=...=b_D=0
$$
and therefore $f=0$ against the assumptions.
So the case $D+\kappa=l-1$ cannot occur.

\item
If $D+\kappa>\ell-1$, then $\kappa> D> D+\kappa-\ell$, since $\kappa>D$ and $\kappa-\ell<0$.
In this case we shall rather consider the polynomial $\wtilde H(q)$ defined by:
$$
\wtilde H(q) =\sum_{i=0}^D a_i q^i -
\sum _{i=\kappa}^{\ell-1}b_{i-\kappa}q^i -
\sum_{i=0}^{D+\kappa-\ell}b_{i+\ell-\kappa}q^i.
$$
Notice that $H(\zeta_\ell)=\wtilde H(\zeta_\ell)=0$ and that
$\wtilde H(q)$ has degree smaller or equal than $\ell-1$.
As in the previous case, $\wtilde H(q)$ is a $k$-multiple of $\chi(q)$. We get
$$
b_j=0\hbox{~for~}j=0,...,\ell-1-\kappa
$$
and
$$
a_0-b_{\ell-\kappa}=...=a_{D+\kappa-\ell}-b_D=a_{D+\kappa-\ell+1}=...=a_D=0.
$$
We conclude that $f(q)=q^{\kappa-\ell}$.
\end{enumerate}
This ends the proof.
\end{proof}

We are going to deduce Proposition \ref{prop:ratexp} from Lemma \ref{lemma:rationalfunctions}
in two steps: first of all we are going to show that we can drop
the assumption that $[k:k_0]$ is finite and then that one
can always reduce to the case of a rational function.

\begin{lemma}\label{lemma:ratexp}
Lemma \ref{lemma:rationalfunctions} holds if $k/k_0$ is a finitely generated
(not necessarily algebraic) extension.
\end{lemma}

\begin{rmk}
Since $f(q)\in k(q)$, replacing $k$ by the field generated by the coefficients of $f$ over $k_0$, we can always assume that $k/k_0$ is finitely generated.
\end{rmk}

\begin{proof}
Let $\wtilde k$ be the algebraic closure of $k_0$ in $k$
and let $k^\p$ be an intermediate field of $k/\wtilde k$,
such that $f(q)\in k^\p(q)\subset k(q)$ and that $k^\p/\wtilde k$ has minimal
transcendence degree $\iota$.
We suppose that $\iota>0$, to avoid the situation of Lemma \ref{lemma:rationalfunctions}.
So let $a_1,\dots,a_\iota$ be transcendence basis of
$k^\p/\wtilde k$ and let $k^{\p\p}=\wtilde k(a_1,\dots,a_\iota)$.
If $k^\p/\wtilde k$ is purely transcendental, \ie if $k^\p=k^{\p\p}$,
then $f(q)=P(q)/Q(q)$, where $P(q)$ and $Q(q)$ can
be written in the form:
$$
P(q)=\sum_i \sum_{\ul{j}} \a^{(i)}_{\ul{j}} a_{\ul{j}}q^i
\hskip 10 pt\hbox{and}\hskip 10 pt
Q(q)=\sum_i \sum_{\ul{j}} \be^{(i)}_{\ul{j}} a_{\ul{j}}q^i,
$$
with $\ul{j}=(j_1,\dots,j_\iota)\in\Z_{\geq 0}^\iota$,
$a_{\ul{j}}=a_{j_i}\cdots a_{j_\iota}$ and $\a_{\ul j}^{(i)},\be_{\ul j}^{(i)}\in\wtilde k$.
If we reorganize the terms of $P$ and $Q$ so that
$$
P(q)= \sum_{\ul{j}} a_{\ul{j}} D_{\ul{j}}(q)
\hskip 10 pt\hbox{and}\hskip 10 pt
Q(q)= \sum_{\ul{j}} a_{\ul{j}} C_{\ul{j}}(q),
$$
we conclude that the assumption $f(\zeta_\ell)\subset \mu_\ell$ for infinitely many
primes $\ell$
implies that
$f_{\ul{j}}= \frac{ D_{\ul{j}}}{ C_{\ul{j}}}$
is a rational function with coefficients in $\wtilde k$
satisfying the assumptions of Lemma \ref{lemma:rationalfunctions}.
Moreover, since the $f_j$'s take the same values at infinitely many roots of unity, they are all equal.
Finally, we conclude that $f_{\ul j}(q)=q^d$ for any $\ul j$ and hence that
$f= q^d \frac{\sum \alpha_{\ul{j}}}{\sum \alpha_{\ul{j}}}=q^d$.
\par
Now let us suppose that $k^\p=k^{\p\p}(b)$ for some primitive element $b$, algebraic over
$k^{\p\p}$, of degree $e$.
Then once again we write
$f(q)=P(q)/Q(q)$, with:
$$
P(q)=\sum_i \sum_{h=0}^{e-1} \a_{i,h}b^hq^i
\hskip 10 pt\hbox{and}\hskip 10 pt
Q(q)=\sum_i \sum_{h=0}^{e-1} \be_{i,h}b^hq^i,
$$
with $\a_{i,h},\be_{i,h}\in k^{\p\p}$.
Again we conclude that $\frac{\sum_i \a_{i,h}q^i}{\sum_i \be_{i,h}q^i}=q^d$
for any $h=0,\dots,e-1$, and hence that $f(q)=q^d$.
\end{proof}

\begin{proof}[End of the proof of Proposition \ref{prop:ratexp}]
Let $\wtilde K=k(q,f)\subset K$. If the characteristic of $k$ is $p$,
replacing $f$ by a $p^n$-th power of $f$, we can suppose that
$\wtilde K/k(q)$ is a Galois extension.
So we set:
$$
y=\prod_{\varphi\in Gal(\wtilde K/k(q))}f^\varphi\in k(q).
$$
%Let $v\in \cC_{k(q)}$ and $w_1,\dots,w_r\in\cC_{\wtilde K}$ the places of $\wtilde K$ dividing $v$.
For infinitely many $v\in\cC_{k(q)}$ such that $\kappa_v$ is a prime,
we have $f^{\kappa _v}\equiv 1$ modulo $w$, for any $w\vert v$.
Since $Gal(\wtilde K/K)$ acts transitively over the set of places $w\in\cC_{\wtilde K}$ such that $w\vert v$,
this implies that $y^{\kappa_v}\equiv 1$ modulo $\pi_v$.
Then
Lemmas \ref{lemma:ratexp} and \ref{lemma:rationalfunctions} allow us to conclude that $y\in q^\Z$.
This proves that we are in the following situation: $f$ is an algebraic function
such that $|f|_w=1$ for any $w\in\cP_{\wtilde K,f}$ and that $|f|_w\neq 1$ for any $w\in\cP_{\wtilde K,\infty}$.
We conclude that $f=c q^{s/r}$ for some nonzero integers $s,r$ and some constant $c$ in a finite extension of
$k$.
Since $f^{\kappa_v}\equiv 1$ modulo $w$ for all $w\in\cC_{\tilde K}$ such that $\kappa_v\in\wp$,
we finally obtain that $r=1$ and $c=1$.
\end{proof}

%%%%%%%%%%%%%%%%%%%%%%%%%%%%%%%%%%%%%%%%%%%%%%%%%%%%%%%%%%%%%%%%%%%%
%%%%%%%%%%%%%%%%%%%%%%%%%%%%%%%%%%%%%%%%%%%%%%%%%%%%%%%%%%%%%%%%%%%%
%%%%%%%%%%%%%%%%%%%%%%%%%%%%%%%%%%%%%%%%%%%%%%%%%%%%%%%%%%%%%%%%%%%%
\section{Main result}
\label{sec:trivsol}
%%%%%%%%%%%%%%%%%%%%%%%%%%%%%%%%%%%%%%%%%%%%%%%%%%%%%%%%%%%%%%%%%%%%
%%%%%%%%%%%%%%%%%%%%%%%%%%%%%%%%%%%%%%%%%%%%%%%%%%%%%%%%%%%%%%%%%%%%
%%%%%%%%%%%%%%%%%%%%%%%%%%%%%%%%%%%%%%%%%%%%%%%%%%%%%%%%%%%%%%%%%%%%

In this section we are proving an analogue of the Grothendieck conjecture
on $p$-curvatures under the assumption that $q$ is transcendental. Roughly speaking, we are going to prove that a
$q$-difference module is trivial if and only if its reduction modulo almost all
cyclotomic places is trivial.

\medskip
We say that the $q$-difference module $\cM=(M,\Sgq)$ of rank $\nu$
over a $q$-difference field $\cF$ is trivial
if there exists a basis $\ul f$ of $M$ over $\cF$ such that $\Sgq\ul f=\ul f$.
This is equivalent to ask that the $q$-difference system
associated to $\cM$ with respect to a basis (and hence any basis)
$\ul e$ has a fundamental solution
in $\GL_\nu(\cF)$.
We say that a $q$-difference module $\cM=(M,\Sgq)$ over $\cA$ becomes trivial over
a $q$-difference field $\cF$
over $\cA$ if the $q$-difference module $(M\otimes_\cA\cF,\Sgq\otimes\sgq)$ is trivial.

\begin{thm}\label{thm:GrothKaz}
A $q$-difference module $\cM$ over $\cA$ has zero
$\kappa_v$-curvature modulo $\phi_v$ for almost all $v\in\cC$ if and only if
$\cM$ becomes trivial over $K(x)$.
\end{thm}

\begin{rmk}
As proved in \cite[Proposition 2.1.2]{DVInv}, if $\Sgq^{\kappa_v}$ is the identity
modulo $\phi_v$ then the $q_v$-difference module $\cM\otimes_\cA\cA/(\phi_v)$
is trivial.
\end{rmk}

Theorem \ref{thm:GrothKaz} is equivalent to the following statement,
which is a $q$-analogue of the conjecture stated at the very end of
\cite{MatzatPutCrelle}:

\begin{cor}\label{cor:Grothit}
For a $q$-difference module $\cM$ over $\cA$ the following statement are equivalent:
\begin{enumerate}
\item
The $q$-difference module $\cM$ over $\cA$
becomes trivial over $K(x)$;

\item
It induces an iterative
$q_v$-difference structure over $\cM_{k_v(x)}$ for almost all $v\in\cC$;

\item
It induces a trivial iterative
$q_v$-difference structure over $\cM_{k_v(x)}$ for almost all $v\in\cC$.

\end{enumerate}
\end{cor}

\begin{rmk}
The first assertion is equivalent to the fact that the Galois group of $\cM_{K(x)}$ is trivial,
while the fourth assertion is equivalent to the fact that the
iterative Galois group of $\cM_{k_v(x)}$ over $k_v(x)$ is $1$ for almost all $v\in\cC$.
\end{rmk}

\begin{proof}
The equivalence $1\Leftrightarrow 2$ is a consequence of Proposition \ref{prop:iteratedred} and Theorem \ref{thm:GrothKaz},
while the implication $3\Rightarrow 2$ is tautological.
\par
Let us prove that $1\Rightarrow 3$.
If the $q$-difference module $\cM$ becomes trivial over $K(x)$, then
there exist an $\cA$-algebra $\cA^\p$, of the form \eqref{eq:algebraA}, obtained
from $\cA$ inverting a polynomial and its $q$-iterates, and a basis $\ul e$ of
$M\otimes_\cA\cA^\p$ over $\cA^\p$, such that the associated
$q$-difference system is $\sgq(Y)=Y$.
Therefore, for almost all $v \in \cC$, $\cM$ induces an iterative $q_v$-difference module $\cM_{k_v(x)}$
whose iterative $q_v$-difference equations are given by $\frac{d_{q_v}^{\kappa_v}}{[\kappa_v]_{q_v}^!}(Y)=0$
for all $n \in \mathbb{N}$
(\cf \cite[Proposition 3.17]{HardouinIterative}).
\end{proof}

As far as the proof of Theorem \ref{thm:GrothKaz} is regarded,
one implication is trivial. The proof of the other
is divided into steps. So let us suppose that the
$q$-difference module $\cM$ over $\cA$ has zero
$\kappa_v$-curvature modulo $\phi_v$ for almost all $v\in\cC$, then:
\begin{trivlist}
\item{\it Step 1.}
The $q$-difference module $\cM$ becomes trivial over $K((x))$, meaning that
the module $\cM_{K((x))}=(M\otimes_\cA K((x)),\Sgq\otimes\sgq)$ is trivial
(\cf Corollary \ref{cor:FormalSol} below).

\item{\it Step 2.}
There exists a basis $\ul e$ of $\cM_{K(x)}$,
such that the associated $q$-difference system has a fundamental matrix of solution
$Y(x)$ in $\GL(K[[x]])$ whose entries are
Taylor expansions of rational functions (\cf Proposition \ref{prop:boreldwork} below).
\end{trivlist}

\begin{rmk}
Theorem \ref{thm:GrothKaz} is a function field analogue of the main result of
\cite{DVInv}.
Step 1 is inspired by \cite[13.1]{KatzTurrittin} (\cf also
\cite[\S6]{DVInv} for $q$-difference equations over number fields).
The main difference is Proposition \ref{prop:ratexp} proved above.
Step 2 is closed to \cite[\S8]{DVInv} and uses
the Borel-Dwork criteria
(\cf \cite[VIII, 1.2]{AGfunctions}).
\end{rmk}

\subsubsection*{Step 1: triviality over $K((x))$}

The triviality over $K((x))$ is a consequence of Theorem \ref{thm:FormalSol}:

\begin{cor}\label{cor:FormalSol}
If there exists an infinite set of positive primes $\wp\subset\Z$
such that the $q$-difference module $\cM$ over $\cA$ has zero
$\kappa_v$-curvature modulo $\pi_v$
(and \emph{a fortiori} modulo $\phi_v$) for all $v\in\cC$ with
$\kappa_v\in\wp$, then
$\cM$ becomes trivial over the field of formal Laurent series
$K((x))$.
\end{cor}

\begin{proof}
If $\cM$ has zero $\kappa_v$-curvature modulo $\pi_v$
then (\cf \eqref{eq:nilpotence} for notation) we
actually have:
$$
\hbox{for all $v\in\cC$ such that $\kappa_v\in\wp$,
$D_0^{\kappa_v}\equiv 1$ and $N_0^{\kappa_v}\equiv 1$ modulo $\pi_v$,}
$$
where $\Sgq\ul e=\ul e A_0$, for a chosen basis $\ul e$ of $\cM_{K((x))}$ and
a constant matrix $A_0 = D_{0} N_{0} \in \GL_\nu(K)$.
This immediately implies, because of Proposition \ref{prop:ratexp},
that all the exponents are in $q^\Z\subset k(q)\subset K$ and that the
matrix $A_0$ of $\cM$, w.r.t. the $K(x)$-basis $\ul e$, is diagonalisable.
Therefore there exist a diagonal matrix $D$ with coefficients in $\Z$ and a
matrix $C\in \GL_\nu(K)$ such that the basis $\ul e^\p=\ul e C x^D$ of $\cM_{K((x))}$ is
invariant under the action of $\Sgq$.
\end{proof}

\subsubsection*{Step 2: rationality of solutions}

\begin{prop}\label{prop:boreldwork}
If a $q$-difference module $\cM$ over $\cA$ has zero
$\kappa_v$-curvature modulo $\phi_v$ for almost all $v\in\cC$ then
there exists a basis $\ul e$ of $M_{K(x)}$ over $K(x)$ such that
the associated $q$-difference system has a formal fundamental solution
$Y(x)\in \GL_\nu(K((x)))$, which is the Taylor expansion at $0$ of
a matrix in $\GL_\nu(K(x))$,
\ie $\cM$ becomes trivial over $K(x)$.
\end{prop}

\begin{rmk}
This is the only part of the proof of Theorem \ref{thm:GrothKaz} where we need to suppose that
the $\kappa_v$-curvature are zero
modulo $\phi_v$ \emph{for almost all $v$}.
\end{rmk}

\begin{proof}(\cf \cite[Proposition 8.2.1]{DVInv})
Let $\ul e$ be a basis of $M$ over $K(x)$.
Because of Corollary \ref{cor:FormalSol}, applying a basis change
with coefficients in $K\l[x,\frac 1x\r]$, we can actually suppose that
$\Sgq\ul e=\ul e A(x)$, where $A(x)\in \GL_\nu(K(x))$ has no pole at $0$ and
$A(0)$ is the identity matrix. In the notation of \S\ref{subsec:submodules},
the recursive relation defining the matrices $G_n(x)$ implies that
they have no pole at $0$. This means that $Y(x):=\sum_{n\geq 0}G_{[n]}(0)x^n$
is a fundamental solution of the $q$-difference system
associated to $\cM_{K(x)}$ with respect to the basis $\ul e$,
whose entries verify the following properties:
\begin{itemize}

\item
For any $v\in\cP_\infty$, the matrix $Y(x)$ has infinite $v$-adic radius of meromorphy.
This assertion is a general fact about regular singular $q$-difference systems with $|q|_v\neq 1$.
The proof is based on the estimate of the growth
of the $q$-factorials compared to the growth of $G_n(0)$, which gives the analyticity at $0$,
and on the fact that the
$q$-difference system itself
gives a meromorphic continuation of the solution.

\item
Since $\l|[n]_q\r|_{v,Gauss}=1$ for any noncyclotomic place $v\in\cP_f$,
we have $\l|G_{[m]}(x)\r|_{v,Gauss}\leq 1$ for almost all $v \in \cP_f\setminus \cC$.
For the finitely many $v\in\cP_f$ such that $\l|G_1(x)\r|_{v,Gauss}>1$, there exists a
constant $C>0$ such that $\l|G_{[m]}(x)\r|_{v,Gauss}\leq C^m$, for any positive integer $m$.

\item
For almost all $v\in\cC$ and all positive integer $m$,
$\l|G_{[m]}(x)\r|_{v,Gauss}\leq 1$ (\cf Proposition \ref{prop:iteratedred}), while for
the remaining finitely many $v\in\cC$ there exists a constant $C>0$ such
that $\l|G_{[m]}(x)\r|_{v,Gauss}\leq C^m$ for any positive integer $m$.
\end{itemize}
This implies that:
$$
\limsup_{m\to\infty}\frac{1}{m}
\sum_{v\in\cP}\log^+\l|G_{[m]}(x)\r|_{v,Gauss}=\limsup_{m\to\infty}\frac{1}{m}
\sum_{v\in\cC}\log^+\l|G_{[m]}(x)\r|_{v,Gauss}<\infty.
$$
To conclude that $Y(x)$ is the expansion at zero of a matrix with rational entries
we apply a simplified form of the Borel-Dwork criteria for function fields, which says exactly that
a formal power series having positive radius of convergence for almost all places and
infinite radius of meromorphy at one fixed place is the expansion of a rational function.
The proof in this case is a slight simplification of \cite[Proposition  8.4.1]{DVInv}\footnote{The simplification
comes from the fact that there are no archimedean norms in this setting.},
which is itself a simplification of the more general criteria \cite[Theorem 5.4.3]{Andregrothendieck}.
We are omitting the details.
\end{proof}

%%%%%%%%%%%%%%%%%%%%%%%%%%%%%%%%%%%%%%%%%%%%%%%%%%%%%%%%%
\section{Generic Galois group}
\label{sec:genericgaloisgroup}
%%%%%%%%%%%%%%%%%%%%%%%%%%%%%%%%%%%%%%%%%%%%%%%%%%%%%%%%%

Let $\cM=(M,\Sgq)$ be a $q$-difference module of rank $\nu$ over $\cA$, as in the previous sections.
Since $M_{K(x)}=(M_{K(x)},\Sgq)$ is a $q$-difference module over $K(x)$, we can consider
the collection $Constr_{K(x)}(\cM_{K(x)})$ of all $q$-difference modules obtained from $\cM_{K(x)}$ by
algebraic construction. This means that we consider the family of $q$-difference modules
containing $\cM_{K(x)}$ and closed under
direct sum, tensor product, dual, symmetric and antisymmetric products.
For the reader convenience, we remind the definition of the duality and the tensor product, from which we can deduce
all the other algebraic constructions:
\begin{itemize}
\item
The $q$-difference structure on the dual $M_{K(x)}^*$ of $M_{K(x)}$ is defined by:
$$
\langle\Sgq^*(m^*),m\rangle=\sgq\l(\langle m^*,\Sgq^{-1}(m)\rangle\r),
$$
for any $m^*\in M_{K(x)}^*$ and any $m\in M_{K(x)}$.
\item
If $\cN_{K(x)}=(N_{K(x)},\Sgq)$, the $q$-difference structure on the tensor product
$M_{K(x)}\otimes_{K(x)}N_{K(x)}$ is defined by
$$
\Sgq(m\otimes n)=\Sgq(m)\otimes\Sgq(n),
$$
for any $m\in M_{K(x)}$ and any $n\in N_{K(x)}$ (\cf for instance \cite[\S9.1]{DVInv}
or \cite[\S2.1.6]{sauloyfiltration}).
\end{itemize}
We will denote $Constr_{K(x)}(M_{K(x)})$ the collection of algebraic constructions of the
$K(x)$-vector space $M_{K_(x)}$, \ie the collection of underlying vector spaces of the family
$Constr_{K(x)}(\cM_{K(x)})$.
Notice that $\GL(M_{K(x)})$ acts naturally, by functoriality, on any element of $Constr_{K(x)}(M_{K(x)})$.

\begin{defn}
The \emph{generic Galois group\footnote{In \cite{andreens} it is
called the \emph{intrinsic} Galois group of $\cM_{K(x)}$.}
$Gal(\cM_{K(x)},\eta_{K(x)})$ of $\cM_{K(x)}$}
is the subgroup of $\GL(M_{K(x)})$ which is the stabiliser of all the $q$-difference submodules over $K(x)$
of any object in $Constr_{K(x)}(\cM_{K(x)})$.
\end{defn}

The group $Gal(\cM_{K(x)},\eta_{K(x)})$ is a tannakian object.
In fact, the full tensor category $\langle\cM_{K(x)}\rangle^\otimes$ generated by
$\cM_{K(x)}$ in $Diff(K(x), \sgq)$ is naturally a tannakian category, when
equipped with the forgetful functor
$$
\eta:\langle\cM_{K(x)}\rangle^\otimes\longrightarrow\{\hbox{$K(x)$-vector spaces}\}.
$$
The functor $Aut^\otimes(\eta)$ defined over the category of
$K(x)$-algebras is representable by the algebraic group $Gal(\cM_{K(x)},\eta_{K(x)})$.
\par
Notice that in positive characteristic $p$, the group
$Gal(\cM_{K(x)},\eta_{K(x)})$ is not necessarily reduced.
An easy example is given by the equation $y(qx)=q^{1/p}y(x)$, whose generic Galois group is $\mu_p$
(\cf \cite[\S7]{vdPutReversatToulouse}).

\begin{rmk}\label{rmk:noetherianity}
We recall that the Chevalley theorem, that also holds for nonreduced groups
(\cf \cite[II, \S2, n.3, Corollary 3.5]{demazuregabriel}),
ensures that $Gal(\cM_{K(x)},\eta_{K(x)})$ can be defined as the
stabilizer of a rank one submodule (which is not necessarily a $q$-difference module) of a
$q$-difference module contained in an algebraic construction of $\cM_{K(x)}$.
Nevertheless, it is possible to find a line that defines $Gal(\cM_{K(x)},\eta_{K(x)})$
as the stabilizer and that is also a $q$-difference module.
In fact the noetherianity of $\GL(M_{K(x)})$ implies that
$Gal(\cM_{K(x)},\eta_{K(x)})$ is defined as the stabilizer of a finite family
of $q$-difference submodules $\cW_{K(x)}^{(i)}=(W_{K(x)}^{(i)},\Sgq)$ contained in some objects
$\cM_{K(x)}^{(i)}$ of $\langle\cM_{K(x)}\rangle^\otimes$.
It follows that the line
$$
L_{K(x)}=\wedge^{\dim \mathop\oplus_i W_{K(x)}^{(i)}} \l(\mathop\oplus_i W_{K(x)}^{(i)}\r)
\subset \wedge^{\dim \mathop\oplus_i W_{K(x)}^{(i)}} \l(\mathop\oplus_i M_{K(x)}^{(i)}\r)
$$
is a $q$-difference module and defines $Gal(\cM_{K(x)},\eta_{K(x)})$
as a stabilizer (\cf \cite[proof of Proposition 9]{Katzbull}).
\par
In the sequel, we will use the notation
$Stab(W_{K(x)}^{(i)}, i)$ to say that a group is the stabilizer
of the set of vector spaces $\{W_{K(x)}^{(i)}\}_i$.
\end{rmk}

Let $G$ be a closed
algebraic subgroup of $\GL(M_{K(x)})$ such that $G=Stab(L_{K(x)})$, for some line
$L_{K(x)}$ contained in an object $\cW_{K(x)}$ of $\langle\cM_{K(x)}\rangle^\otimes$.
The $\cA$-lattice $M$ of $M_{K(x)}$ determines an $\cA$-lattice $L$ of $L_{K(x)}$ and
an $\cA$-lattice $W$ of $W_{K(x)}$. The latter is the underlying space of a $q$-difference module
$\cW=(W,\Sgq)$ over $\cA$.

\begin{defn}\label{defn:contains}
Let $\wtilde\cC$ be a cofinite subset of $\cC_K$ and $(\La_v)_{v\in\wtilde\cC}$
be a family of $\cA/(\phi_v)$-linear operators acting on $M\otimes_\cA\cA/(\phi_v)$.
We say that \emph{the algebraic group
$G\subset \GL(M_{K(x)})$ contains the operators $\La_v$ modulo $\phi_v$ for almost all $v\in\cC_K$}
if for almost all $v\in\wtilde\cC$ the operator $\La_v$ stabilizes
$L\otimes_\cA\cA/(\phi_v)$ inside $W\otimes_\cA\cA/(\phi_v)$:
$$
\La_v\in Stab_{\cA/(\phi_v)}(L\otimes_\cA\cA/(\phi_v)).
$$
\end{defn}

\begin{rmk}
As in \cite[10.1.2]{DVInv}, one can prove that the definition above is independent of the choice of
$\cA$, $M$ and $L_{K(x)}$.
\end{rmk}

The main result of this section is the following:

\begin{thm}\label{thm:genGalois}
The algebraic group $Gal(\cM_{K(x)},\eta_{K(x)})$ is the smallest closed algebraic subgroup of
$\GL(M_{K(x)})$ that contains the operators $\Sgq^{\kappa_v}$ modulo
$\phi_v$, for almost all $v\in\cC$.
\end{thm}

\begin{rmk}
The noetherianity of $\GL(M_{K(x)})$ implies that the smallest closed algebraic subgroup of
$\GL(M_{K(x)})$ that contains the operators $\Sgq^{\kappa_v}$ modulo
$\phi_v$, for almost all $v\in\cC$, is well-defined.
\end{rmk}

A part of Theorem \ref{thm:genGalois} is easy to prove:

\begin{lemma}\label{lemma:genGalois}
The algebraic group $Gal(\cM_{K(x)},\eta_{K(x)})$ contains the operators
$\Sgq^{\kappa_v}$ modulo $\phi_v$ for almost all $v\in\cC_K$.
\end{lemma}

\begin{proof}
The statement follows immediately from the fact that $Gal(\cM_{K(x)},\eta_{K(x)})$
can be defined as the stabilizer of a rank one $q$-difference module
in $\langle\cM_{K(x)}\rangle^\otimes$, which is \emph{a fortiori} stable by the action
of $\Sgq^{\kappa_v}$.
\end{proof}

\begin{cor}\label{cor:trivialgroup}
$Gal(\cM_{K(x)},\eta_{K(x)})=\{1\}$ if and only if
$\cM_{K(x)}$ is a trivial $q$-difference module.
\end{cor}

\begin{proof}
Because of the lemma above,
if $Gal(\cM_{K(x)},\eta_{K(x)})=\{1\}$ is the trivial group, then $\Sgq^{\kappa_v}$
induces the identity on $M\otimes_\cA\cA/(\phi_v)$. Therefore Theorem \ref{thm:GrothKaz}
implies that $\cM_{K(x)}$ is trivial.
On the other hand, if $\cM_{K(x)}$ is trivial, then it is isomorphic to the
$q$-difference module $(K^\nu\otimes_K K(x),1\otimes \sgq)$.
It follows that the generic Galois group
$Gal(\cM_{K(x)},\eta_{K(x)})$ is forced to stabilize
all the lines generated by vectors of the type $v\otimes 1$, with $v\in K^\nu$. Therefore
it is the trivial group.
\end{proof}

Now we are ready to give the proof of Theorem \ref{thm:genGalois}, whose main ingredient is Theorem
\ref{thm:GrothKaz}. The argument
is inspired by \cite[\S X]{Katzbull}.

\begin{proof}[Proof of Theorem \ref{thm:genGalois}]
Lemma \ref{lemma:genGalois} says that $Gal(\cM_{K(x)},\eta_{K(x)})$ contains the smallest subgroup $G$
of $\GL(M_{K(x)})$ that contains the operator
$\Sgq^{\kappa_v}$ modulo $\phi_v$ for almost all $v\in\cC_K$.
Let $L_{K(x)}$ be a line contained in some object of the category $\langle\cM_{K(x)}\rangle^\otimes$,
that defines $G$ as a stabilizer. Then there exists a smaller $q$-difference
module $\cW_{K(x)}$ over $K(x)$ that contains $L_{K(x)}$. Let $L$ and $\cW=(W,\Sgq)$
be the associated $\cA$-modules.
Any generator $m$ of $L$ as an $\cA$-module is a cyclic vector for $\cW$ and the operator $\Sgq^{\kappa_v}$
acts on $W\otimes_\cA\cA/(\phi_v)$ with respect to the basis induced by the cyclic basis
generated by $m$ via a diagonal matrix.
Because of the definition of the $q$-difference structure on the dual module
$\cW^*$ of $\cW$, the group
$G$ can be define as the subgroup of $\GL(M_{K(x)})$ that fixes
a line $L^\p$ in $W^*\otimes W$,
\ie such that $\Sgq^{\kappa_v}$ acts as the identity on $L^\p\otimes_\cA\cA/(\phi_v)$,
for almost all cyclotomic places $v$.
It follows from Theorem \ref{thm:GrothKaz} that
the minimal submodule $\cW^\p$ that contains $L^\p$ becomes trivial over $K(x)$.
Since the tensor category generated by $\cW^\p_{K(x)}$
is contained in $\langle\cM_{K(x)}\rangle^\otimes$,
we have a functorial surjective group morphism
$$
Gal(\cM_{K(x)},\eta_{K(x)})\longrightarrow Gal(\cW^\p_{K(x)},\eta_{K(x)})=\{1\}.
$$
We conclude that $Gal(\cM_{K(x)},\eta_{K(x)})$ acts trivially over $\cW^\p_{K(x)}$, and
therefore that $Gal(\cM_{K(x)},\eta_{K(x)})$ is contained in $G$.
\end{proof}

\begin{cor}
Theorem \ref{thm:GrothKaz} and Theorem \ref{thm:genGalois} are equivalent.
\end{cor}

\begin{proof}
We have seen in the proof above that Theorem \ref{thm:GrothKaz} implies
Theorem \ref{thm:genGalois}. Corollary \ref{cor:trivialgroup} gives the opposite implication.
\end{proof}

%%%%%%%%%%%%%%%%%%%%%%%%%%%%%%%%%%%%%%%%%%%%%%%%%%%%%%%%%%%%%%%%%%%%%%%%%%%
\subsection{Finite generic Galois groups}
%%%%%%%%%%%%%%%%%%%%%%%%%%%%%%%%%%%%%%%%%%%%%%%%%%%%%%%%%%%%%%%%%%%%%%%%%%%

We deduce from Theorem \ref{thm:genGalois}
the following description of a finite generic Galois group:

\begin{cor}\label{cor:finitegroups}
The following facts are equivalent:
\begin{enumerate}
\item
There exists a positive integer $r$ such that the
$q$-difference module $\cM=(M,\Sgq)$ becomes trivial
as a $\wtilde q$-difference module over $K(\wtilde q,t)$,
with $\wtilde q^r=q$, $t^r=x$.

\item
There exists a positive integer $r$ such that for almost
all $v\in\cC$ the morphism $\Sgq^{\kappa_v r}$ induces the identity
on $M\otimes_\cA\cA/(\phi_v)$.

\item
There exists a $q$-difference field extension $\cF/K(x)$ of finite degree such that
$\cM$ becomes trivial over $\cF$.

\item
The (generic) Galois group of $\cM$ is finite.
\end{enumerate}
In particular, if $Gal(\cM_{K(x)},\eta_{K(x)})$ is finite, it is necessarily cyclic (of order $r$, if one chooses
$r$ minimal in the assertions above).
\end{cor}

\begin{proof}
The equivalence ``$1\Leftrightarrow 2$'' follows from Theorem \ref{thm:GrothKaz}
applied to the $\wtilde q$-dif\-fer\-ence module $(M\otimes K(\wtilde q,t),\Sgq\otimes\sg_{\wtilde q})$,
over the field $K(\wtilde q,t)$.
\par
If the
generic Galois group
is finite, the reduction modulo $\phi_v$ of $\Sgq^{\kappa_v}$ must be a cyclic operator
of order dividing the cardinality of $Gal(\cM_{K(x)},\eta_{K(x)})$. So we have proved that
``$4\Rightarrow 2$''.
On the other hand, assertion $2$ implies that
there exists a basis of $M_{K(x)}$ such that the representation of $Gal(\cM_{K(x)},\eta_{K(x)})$
is given by the group of diagonal matrices, whose diagonal entries are $r$-th roots of unity.
\par
Of course, assertion $1$ implies assertion $3$. The inverse implication follows from
the proposition below, applied to a cyclic vector of $\cM_{K(x)}$.
\end{proof}

\begin{lemma}
Let $K$ be a field and $q$ an element of $K$ which is not a root of unity. We suppose that
there exists a norm $|~|$ over $K$ such that $|q|\neq 1$\footnote{This assumption is always verified
if $K$ is a finite extension of a field of rational functions $k(q)$, as in this paper, or
if there exists an immersion of $K$ in $\C$.}
and we consider a linear $q$-difference equations
\beq\label{eq:algebraicsolutions}
a_\nu(x)y(q^\nu x)+a_{\nu-1}(x)y(q^{\nu-1} x)+\dots+a_0(x)y(x)=0
\eeq
with coefficients in $K(x)$. If there exists
an algebraic $q$-difference extension $\cF$ of $K(x)$ containing a solution $f$ of
\eqref{eq:algebraicsolutions},
then $f$ is contained in an extension of $K(x)$ isomorphic to $K(\wtilde q, t)$,
with $\wtilde q^r=1$ and $t^r=x$.
\end{lemma}

\begin{proof}
Let us look at \eqref{eq:algebraicsolutions} as an equation with coefficients
in $K((x))$. Then the algebraic solution $f$ of \eqref{eq:algebraicsolutions}
can be identified to a Laurent series in $\ol K((t))$, where $\ol K$ is the algebraic closure of $K$
and $t^r=x$, for a convenient positive integer $r$.
Let $\wtilde q$ be an element of $\ol K$ such that $\wtilde q^r=q$ and that
$\sgq(f)=f(\wtilde q t)$. We can look at \eqref{eq:algebraicsolutions}
as a $\wtilde q$-difference equation with coefficients in $K(\wtilde q,t)$.
Then the recurrence relation induced by \eqref{eq:algebraicsolutions}
over the coefficients of a formal solution shows that there exist
$f_1,\dots,f_s$ solutions of \eqref{eq:algebraicsolutions} in $K(\wtilde q)((t))$
such that $f\in\sum_i\ol K f_i$.
It follows that there exists a finite extension $\wtilde K$ of $K(\wtilde q)$
such that $f\in\wtilde K((t))$.
\par
We fix an extension of $|~|$ to $\wtilde K$, that we still call $|~|$.
Since $f$ is algebraic, it is a germ of meromorphic function at $0$.
Since $|\wtilde q|\neq 1$,
the functional equation \eqref{eq:algebraicsolutions} itself allows
to show that $f$ is actually a meromorphic function with infinite radius of meromorphy.
Finally, $f$ can have at worst a pole at $t=\infty$, since it is an algebraic function,
which actually implies that $f$ is the Laurent expansion of a rational function in
$K(\wtilde q,t)$.
\end{proof}

%%%%%%%%%%%%%%%%%%%%%%%%%%%%%%%%%%%%%%%
\subsection{Devissage of nonreduced generic Galois groups}
%%%%%%%%%%%%%%%%%%%%%%%%%%%%%%%%%%%%%%%

Independently of the characteristic of the base field,
there is no proper Galois correspondence for generic Galois groups.
If $\cN=(N,\Sgq)$ is an object of $\langle\cM_{K(x)}\rangle^\otimes$, then
there exists a normal subgroup $H$ of $Gal(\cM_{K(x)},\eta_{K(x)})$
such that $H$ acts as the identity on $N_{K(x)}$ and
\beq\label{eq:Galoiscorrespondence}
Gal(\cN_{K(x)},\eta_{K(x)})\cong
\frac{Gal(\cM_{K(x)},\eta_{K(x)})}{H}.
\eeq
In fact, the category $\langle\cN_{K(x)}\rangle^\otimes$ is a full subcategory
of $\langle\cM_{K(x)}\rangle^\otimes$ and therefore there exists a surjective functorial
morphism
$$
Gal(\cM_{K(x)},\eta_{K(x)})\longrightarrow
Gal(\cN_{K(x)},\eta_{K(x)}).
$$
The kernel of such morphism is the normal subgroup of $Gal(\cM_{K(x)},\eta_{K(x)})$ that acts as the identity
of $\cN_{K(x)}$.
On the other hand, if $H$ is a normal subgroup of $Gal(\cM_{K(x)},\eta_{K(x)})$,
it is not always possible to find an object $\cN_{K(x)}=(N_{K(x)},\Sgq)$ of $\langle\cM_{K(x)}\rangle^\otimes$
such that we have \eqref{eq:Galoiscorrespondence}.
This happens because the generic Galois group $Gal(\cM_{K(x)},\eta_{K(x)})$ stabilizes
all the sub-$q$-difference modules of the constructions on $\cM_{K(x)}$ but also
other submodules, which are not stable by $\Sgq$. So, if $H=Stab(L_{K(x)})$, for some line
$L_{K(x)}$ in some algebraic construction of $\cM_{K(x)}$,
the orbit of $L_{K(x)}$ with respect to $Gal(\cM_{K(x)},\eta_{K(x)})$ could be a
$q$-difference module, allowing to establish
\eqref{eq:Galoiscorrespondence}, but in general it won't be.

\medskip
In spite of the fact that in this setting we do not have a Galois correspondence, we can
establish some devissage of $Gal(\cM_{K(x)},\eta_{K(x)})$, when it is not reduced.
So let us suppose that the group $Gal(\cM_{K(x)},\eta_{K(x)})$ is nonreduced, and therefore that
the characteristic of $k$ is $p>0$.
Then there exists a maximal reduced subgroup $Gal_{red}(\cM_{K(x)},\eta_{K(x)})$ of
$Gal(\cM_{K(x)},\eta_{K(x)})$
and a short exact sequence of groups:
\beq\label{eq:sesRed}
1\longrightarrow Gal_{red}(\cM_{K(x)},\eta_{K(x)})
\longrightarrow Gal(\cM_{K(x)},\eta_{K(x)})
\longrightarrow \mu_{p^\ell}\longrightarrow 1,
\eeq
for some positive integer $\ell$, uniquely determined by the above short exact sequence.
We remind that the subgroup $Gal_{red}(\cM_{K(x)},\eta_{K(x)})$ of
$Gal(\cM_{K(x)},\eta_{K(x)})$ is normal.

\begin{thm}\label{thm:genGaloisred}
The subgroup $Gal_{red}(\cM_{K(x)},\eta_{K(x)})$ of $Gal(\cM_{K(x)},\eta_{K(x)})$
is the smallest algebraic subgroup of
$\GL(M_{K(x)})$ whose reduction modulo $\phi_v$ contains the operators
$\Sgq^{\kappa_v p^\ell}$ for almost all $v\in\cC_K$.
\end{thm}

We first prove two lemmas.

\begin{lemma}\label{lemma:devissage1}
The group $Gal_{red}(\cM_{K(x)},\eta_{K(x)})$ is contained in the smallest algebraic subgroup
$H$ of
$\GL(M_{K(x)})$ whose reduction modulo $\phi_v$ contains the operators
$\Sgq^{\kappa_v p^\ell}$ for almost all $v\in\cC_K$.
\end{lemma}

\begin{proof}
Let $H$ be the smallest algebraic subgroup of
$\GL(M_{K(x)})$ whose reduction modulo $\phi_v$ contains the operators
$\Sgq^{\kappa_v p^\ell}$ for almost all $v\in\cC_K$.
We know that $H=Stab(L_{K(x)})$ for some line $L_{K(x)}$ contained in some object of
$\langle\cM_{K(x)}\rangle^\otimes$.
Once again, as in the proof of Theorem \ref{thm:genGalois}, we can find another line $L^\p_{K(x)}$, that
defines $H$ as a stabilizer and which is actually fixed by $H$. It follows that $L^\p$ generates
a $q$-difference module $\cW^\p$ over $\cA$,
that satisfies the hypothesis of
Corollary \ref{cor:finitegroups}.
We conclude that there exists a nonnegative integer
$\ell^\p\leq\ell$ such that $H$ is contained in the
kernel of the surjective map:
\beq\label{eq:ell'}
Gal(\cM_{K(x)},\eta_{K(x)})\longrightarrow Gal(\cW^\p_{K(x)},\eta_{K(x)})=\mu_{p^{\ell^\p}},
\eeq
and therefore that $Gal_{red}(\cM_{K(x)},\eta_{K(x)})\subset H$.
\end{proof}

\begin{lemma}\label{lemma:devissage2}
Let $q^{(\ell)}=q^{p^\ell}$. We consider the $q^{(\ell)}$-difference module $\cM_{K(x)}^{(\ell)}$
obtained from $\cM_{K(x)}$ iterating $\Sgq$, \ie
$\cM_{K(x)}^{(\ell)}=(M_{K(x)},\Sg_{q^{(\ell)}})$, with $\Sg_{q^{(\ell)}}=\Sgq^{p^\ell}$.
Then $Gal(\cM_{K(x)}^{(\ell)},\eta_{K(x)})$ is the smallest algebraic subgroup of
$\GL(M_{K(x)})$ whose reduction modulo $\phi_v$ contains the operators
$\Sgq^{\kappa_v p^\ell}$ for almost all $v\in\cC_K$.
\end{lemma}

\begin{proof}
Since the characteristic of $k$ is $p>0$, the order $\kappa_v$ of $q_v$ in the residue field $k_v$
is a divisor of $p^n-1$ for some positive integer $n$.
It follows that the order of $q^{(\ell)}$ modulo $v$ is equal to $\kappa_v$ for almost all $v\in\cC_K$.
Theorem \ref{thm:genGalois} allows to conclude, since
$\Sg_{q^{(\ell)}}^{\kappa_v}=\Sgq^{\kappa_v p^\ell}$.
\end{proof}

\begin{proof}[Proof of Theorem \ref{thm:genGaloisred}]
We will prove the statement by induction on $\ell\geq 0$, in the short exact sequence \eqref{eq:sesRed}.
The statement is trivial for $\ell=0$, since in this case
$Gal_{red}(\cM_{K(x)},\eta_{K(x)})=Gal(\cM_{K(x)},\eta_{K(x)})$.
Let us suppose that $\ell>0$ and that the statement is proved for any $\ell^\p<\ell$.
In the notation of the lemmas above, we have:
$$
Gal_{red}(\cM_{K(x)},\eta_{K(x)})\subset H.
$$
We suppose that the inclusion is strict, otherwise there would be nothing to prove.
This means that in \eqref{eq:ell'} we have $\ell^\p>0$:
$$
\xymatrix{H\ar@{^{(}->}[r]& Gal(\cM_{K(x)},\eta_{K(x)})\ar@{->>}[r]&\mu_{p^{\ell^\p}}.}
$$
We claim that $H$ is the smallest subgroup that contains
$\Sgq^{\kappa_v p^{\ell^\p}}$ modulo $\phi_v$ for almost all $v$
and therefore that $H=Gal(\cM_{K(x)}^{(\ell^\p)},\eta_{K(x)})$, because of Lemma
\ref{lemma:devissage2}.
In fact the smallest subgroup that contains
$\Sgq^{\kappa_v p^{\ell^\p}}$ modulo $\phi_v$ for almost all $v$
is contained in $H$
by definition, while morphism \eqref{eq:ell'} proves that $\Sgq^{\kappa_v p^{\ell^\p}}$
stabilizes the line $L_{K(x)}$, considered in Lemma \ref{lemma:devissage1}, modulo $\phi_v$.
Then Lemma \ref{lemma:devissage2} implies that $H=Gal(\cM_{K(x)}^{(\ell^\p)},\eta_{K(x)})$.
\par
Since $Gal_{red}(\cM_{K(x)},\eta_{K(x)})=Gal_{red}(\cM_{K(x)}^{(\ell^\p)},\eta_{K(x)})\subset H$,
we have a short exact sequence:
$$
1\longrightarrow Gal_{red}(\cM_{K(x)}^{(\ell^\p)},\eta_{K(x)})
\longrightarrow Gal(\cM_{K(x)}^{(\ell^\p)},\eta_{K(x)})
\longrightarrow \mu_{p^{\ell-\ell^\p}}
\longrightarrow 1.
$$
The inductive hypotheses implies that
$Gal_{red}(\cM_{K(x)}^{(\ell^\p)},\eta_{K(x)})$ is the smallest subgroup of $\GL(M_{K(x)})$
containing the operators $\Sg_{q^{(\ell)}}^{\kappa_v p^{\ell-\ell^\p}}=\Sgq^{\kappa_v p^\ell}$.
This ends the proof.
\end{proof}

We obtain the following corollary:

\begin{cor}\label{cor:diffgenGaloisred}
In the notation of the theorem above:
\begin{itemize}

\item
$Gal_{red}(\cM_{K(x)},\eta_{K(x)})=Gal(\cM_{K(x)}^{(\ell)},\eta_{K(x)})$.

\item
Let $\wtilde K$ be a finite extension of $K$ containing a $p^\ell$-th root $q^{1/p^\ell}$ of $q$.
Then the generic Galois group $Gal(\cM_{\wtilde K(x^{1/p^\ell})},\eta_{\wtilde K(x^{1/p^\ell})})$
of the $q^{1/p^\ell}$-difference module $\cM_{\wtilde K(x^{1/p^\ell})}$ is reduced
and
$$
Gal(\cM_{\wtilde K(x^{1/p^\ell})},\eta_{\wtilde K(x^{1/p^\ell})})
\subset Gal_{red}(\cM_{K(x)},\eta_{K(x)})
\otimes_{K(x)}\wtilde K(x^{1/p^\ell}).
$$

\end{itemize}
\end{cor}

\begin{proof}
The first statement is a rewriting of Lemma \ref{lemma:devissage2}.
We have to prove the second statement.
If $\ul e$ is a basis of $M_{K(x)}$ such that $\Sgq\ul e=\ul e A(x)$,
then in $\cM_{K(x^{1/p^\ell})}=(M_{K(x^{1/p^\ell})},\Sg_{q^{1/p^\ell}}:=\Sgq\otimes\sg_{q^{1/p^\ell}})$
we have:
$$
\Sg_{q^{1/p^\ell}}(\ul e\otimes 1)=(\ul e\otimes 1)A(x).
$$
It follows that the generic Galois group $Gal(\cM_{K(x^{1/p^\ell})},\eta_{K(x^{1/p^\ell})})$ is the smallest
algebraic subgroup of $\GL(M_{K(x^{1/p^\ell})})$ that contains the operators
$\Sg_{q^{1/p^\ell}}^{\kappa p^\ell}=\Sgq^{\kappa_v p^\ell}\otimes 1$.
This proves that
$$
Gal(\cM_{\wtilde K(x^{1/p^\ell})},\eta_{\wtilde K(x^{1/p^\ell})})
\subset
Gal_{red}(\cM_{K(x)},\eta_{K(x)})
\otimes_{K(x)}\wtilde K(x^{1/p^\ell}).
$$
\end{proof}

%%%%%%%%%%%%%%%%%%%%%%%%%%%%%%%%%%%%%%%%%%%%%%%%%%%%%%%%%%%%%%%%%%%%%%%%%%%%%%%
%%%%%%%%%%%%%%%%%%%%%%%%%%%%%%%%%%%%%%%%%%%%%%%%%%%%%%%%%%%%%%%%%%%%%%%%%%%%%%%
\section{Grothendieck conjecture on $p$-curvatures for $q$-difference modules in characteristic zero, with $q\neq 1$}
\label{sec:modulescarzero}
%%%%%%%%%%%%%%%%%%%%%%%%%%%%%%%%%%%%%%%%%%%%%%%%%%%%%%%%%%%%%%%%%%%%%%%%%%%%%%%
%%%%%%%%%%%%%%%%%%%%%%%%%%%%%%%%%%%%%%%%%%%%%%%%%%%%%%%%%%%%%%%%%%%%%%%%%%%%%%%

Let $K$ be a a finitely generated extension of $\Q$ and $q\in K\smallsetminus\{0,1\}$.
The previous results, combined with an improved version of
\cite{DVInv}, give a ``curvature'' characterization of the generic Galois group of a
$q$-difference module over $K(x)$.
We will constantly distinguish three cases:
\begin{itemize}
\item
$q$ is a root of unity;
\item
$q$ is transcendental over $\Q$;
\item
$q$ is algebraic over $\Q$, but is not a root of unity.
\end{itemize}

%%%%%%%%%%%%%%%%%%%%%%%%%%%%%%%%%%%%%%%%%%%%%%%%%%%%%%%%%%%%%%%%%%%%%%%%%%%%%%%
\paragraph{Case 1: $q$ is a root of unity.}
%%%%%%%%%%%%%%%%%%%%%%%%%%%%%%%%%%%%%%%%%%%%%%%%%%%%%%%%%%%%%%%%%%%%%%%%%%%%%%%
If $q$ it a primitive root of unity of order $\kappa$,
it is not difficult to prove that:

\begin{prop}[{\cite[Proposition  2.1.2]{DVInv}}]\label{prop:GrothKatzroot}
A $q$-difference module $\cM_{K(x)}$ over $K(x)$ is trivial if and only if
$\Sgq^\kappa$ is the identity.
\end{prop}

%%%%%%%%%%%%%%%%%%%%%%%%%%%%%%%%%%%%%%%%%%%%%%%%%%%%%%%%%%%%%%%%%%%%%%%%%%%%%%%
\paragraph{Case 2: $q$ is transcendental over $\Q$.}
%%%%%%%%%%%%%%%%%%%%%%%%%%%%%%%%%%%%%%%%%%%%%%%%%%%%%%%%%%%%%%%%%%%%%%%%%%%%%%%
If $q$ is transcendental over $\Q$, we can always find an intermediate field $k$ of $K/\Q$
such that $K$ is a finite extension of $k(q)$. We are in the situation of Theorem \ref{thm:GrothKaz}, that we can rephrase
as follows:

\begin{thm}\label{thm:GrothKatztrans}
A $q$-difference module $\cM_{K(x)}=(M_{K(x)},\Sgq)$ over $K(x)$ is trivial if and only if there exists
a $k$-algebra $\cA$ (as in \eqref{eq:algebraA}) and a $\Sgq$-stable $\cA$-lattice $M$ of $M_{K(x)}$ such that
for almost all cyclotomic places $v\in\cC$ the $v$-curvature
$$
\Sgq^{\kappa_v}: M\otimes_\cA\cA/(\phi_v)
\longrightarrow M\otimes_\cA\cA/(\phi_v)
$$
is the identity.
\end{thm}

%%%%%%%%%%%%%%%%%%%%%%%%%%%%%%%%%%%%%%%%%%%%%%%%%%%%%%%%%%%%%%%%%%%%%%%%%%%%%%%
\paragraph{Case 3: $q$ is algebraic over $\Q$, but not a root of unity.}
%%%%%%%%%%%%%%%%%%%%%%%%%%%%%%%%%%%%%%%%%%%%%%%%%%%%%%%%%%%%%%%%%%%%%%%%%%%%%%%
Finally if $q$ is algebraic, but not a root of unity, we are in the following situation.
We call $Q$ the algebraic closure of $\Q$ inside $K$,
$\cO_Q$ the ring of integer of $Q$, $v$ a finite places of $Q$ and $\pi_v$ a $v$-adic uniformizer.
For almost all $v$, the order $\kappa_v$ of $q$ modulo $v$, as a root of unity, and the positive integer power $\phi_v$ of
$\pi_v$, such that $\phi_v^{-1}(1-q^{\kappa_v})$
is a unit of $\cO_Q$, are well defined.
The field $K$ has the form $Q(\ul a,b)$, where $\ul a=(a_1,\dots,a_r)$ is a transcendent basis
of $K/Q$ and $b$ is a primitive element of the algebraic extension $K/Q(\ul a)$.
Choosing conveniently the set of generators $\ul a,b$,
we can always find an algebra $\cA$ of the form:
\beq\label{eq:algebraAalg}
\cA=\cO_Q\l[\ul a,b,x,\frac{1}{P(x)},\frac{1}{P(qx)},...\r],
\eeq
for some $P(x)\in \cO_Q\l[\ul a,b,x\r]$,
and a $\Sgq$-stable $\cA$-lattice $M$ of $\cM_{K(x)}$, so that
we can consider the
linear operator
$$
\Sgq^{\kappa_v}:M\otimes_\cA\cA/(\phi_v)\longrightarrow M\otimes_\cA\cA/(\phi_v),
$$
that we will call the $v$-curvature of $\cM_{K(x)}$-modulo $\phi_v$.
Notice that $\cO_Q/(\phi_v)$ is not an integral domain in general.
\par
We are going to prove the following:

\begin{thm}\label{thm:GrothKatzalg}
A $q$-difference module $\cM_{K(x)}=(M_{K(x)},\Sgq)$ over $K(x)$ is trivial if and only if there exists
a $k$-algebra $\cA$ as above and a $\Sgq$-stable $\cA$-lattice $M$ of $M_{K(x)}$ such that
for almost all finite places $v$ of $Q$ the $v$-curvature
$$
\Sgq^{\kappa_v}: M\otimes_\cA\cA/(\phi_v)
\longrightarrow M\otimes_\cA\cA/(\phi_v)
$$
is the identity.
\end{thm}

The theorem above is proved in \cite{DVInv} under the assumption that $K$ is a number field, \ie that $Q=K$.
Here $K$ is only a finitely generated extension of $\Q$. The proof in this more general case is given in \S\ref{subsec:qalg} below.
Notice that the proofs in \cite{DVInv} require $K$ to be a number field, so that the proofs below relies crucially on
\cite{DVInv}, but is not a generalization of the arguments in \emph{loc. cit.}

%%%%%%%%%%%%%%%%%%%%%%%%%%%%%%%%%%%%%%%%%%%%%%%%%%%%%%%%%%%%%%%%%%%%%%%%%%%%%%%
\paragraph{A unified statement.}
%%%%%%%%%%%%%%%%%%%%%%%%%%%%%%%%%%%%%%%%%%%%%%%%%%%%%%%%%%%%%%%%%%%%%%%%%%%%%%%
In order to give a unified statement for the three theorems above we introduce the
following notation:
\begin{trivlist}

\item $\bullet$
if $q$ is a root of unity of order $\kappa$, we can take $\cC$ to be the set containing
only the trivial valuation $v$ on $K$, $\cA$ to be a $\sgq$-stable extension of $K[x]$ obtained inverting a convenient
polynomial, $(\phi_v)=(0)$ and $\kappa_v=\kappa$;

\item $\bullet$
if $q$ is transcendental, $\cC$ is the set of cyclotomic places as in the notation of the previous sections;

\item $\bullet$
if $q$ is algebraic, not a root of unity, we set $\cC$ to be the set of finite places of $Q$.
\end{trivlist}
Therefore we have:

\begin{thm}\label{thm:car0modules}
A $q$-difference module $\cM_{K(x)}=(M_{K(x)},\Sgq)$ over $K(x)$ is trivial if and only if there exists
a $k$-algebra $\cA$ as above and a $\Sgq$-stable $\cA$-lattice $M$ of $M_{K(x)}$ such that
for any $v$ in a cofinite nonempty subset of $\cC$, the $v$-curvature
$$
\Sgq^{\kappa_v}: M\otimes_\cA \cA/(\phi_v)
\longrightarrow M\otimes_\cA \cA/(\phi_v)
$$
is the identity.
\end{thm}

%%%%%%%%%%%%%%%%%%%%%%%%%%%%%%%%%%%%%%%%%%%%%%%%%%%%%%%%%%%%%%%%%
\subsection{Proof of Theorem \ref{thm:GrothKatzalg}}
\label{subsec:qalg}
%%%%%%%%%%%%%%%%%%%%%%%%%%%%%%%%%%%%%%%%%%%%%%%%%%%%%%%%%%%%%%%%%

We only need to prove Theorem \ref{thm:GrothKatzalg} under the assumption that $K$ is not a number field.
The proof (\cf the two subsections below)
will repose on \cite[Theorem 7.1.1]{DVInv}, which is exactly the same statement plus the extra assumption that $K$ is a number field.

%%%%%%%%%%%%%%%%%%%%%%%%%%%%%%%%%%%%%%%%%%%%%%%%%%%%%%%%%%%%%%%%%
\paragraph{Global nilpotence.}
%%%%%%%%%%%%%%%%%%%%%%%%%%%%%%%%%%%%%%%%%%%%%%%%%%%%%%%%%%%%%%%%%
In this and in the following subsection we assume that:
$$
\hbox{$K$ is a transcendental
finitely generated extension of $\Q$ and $q$ is an algebraic number.}
\leqno{(\cH)}
$$

\begin{prop}\label{prop:globalnilpMS}
Under the hypothesis $(\cH)$, for a $q$-difference module $\cM_{K(x)}=(M_K(x),\Sgq)$ we have:
\begin{enumerate}
\item
If $\Sgq^{\kappa_v}$ induces a unipotent linear morphism on
$M\otimes_\cA\cA/(\pi_v)$ for infinitely many finite places $v$ of $Q$,
then the $q$-difference module $\cM_{K(x)}$ is regular singular.

\item
If there exists a set of finite places $v$ of $Q$ of Dirichlet density 1 such that
$\Sgq^{\kappa_v}$ induces a unipotent linear morphism on
$M\otimes_\cA\cA/(\pi_v)$,
then the $q$-difference module $\cM_{K(x)}$ is regular singular and its exponents at $0$ and $\infty$ are in $q^\Z$.

\item
If $\Sgq^{\kappa_v}$ induces the identity on
$M_\cA\otimes_\cA\cA/(\pi_v)$ for almost all finite places $v$ of $Q$ in a set of Dirichlet density $1$,
then the $q$-difference modules $\cM_{K((x))}$ and $\cM_{K((1/x))}$ are trivial.
\end{enumerate}
\end{prop}

We recall that a subset $S$ of the set of finite places $\cC$ of $Q$ has Dirichlet density 1 if
\beq\label{eq:density}
\limsup_{s\to 1^+}
\frac{\sum_{v\in S,v\vert p}p^{-sf_v}}{\sum_{v\in S_f,v\vert p}p^{-sf_v}}=1,
\eeq
where $f_v$ is the degree of the residue field of $v$ over $\mathbb F_p$.

\begin{proof}
The proof is the same as \cite[Theorem 6.2.2 and Proposition 6.2.3]{DVInv} (\cf also Theorem \ref{thm:FormalSol}
and Corollary \ref{cor:FormalSol} above).
The idea is that one has to choose a basis
$\ul e$ of $M_\cA$ such that
$\Sgq\ul e=\ul e A(x)$ for some $A(x)\in \GL_\nu(\cA)$.
Then the hypothesis on the reduction of $\Sgq^{\kappa_v}$ modulo
$\pi_v$ forces $A(x)$ not to have poles at $0$ and $\infty$.
Moreover we deduce that $A(0),A(\infty)\in \GL_\nu(K)$ are actually
semisimple matrices, whose eigenvalues are in $q^\Z$.
\end{proof}

%%%%%%%%%%%%%%%%%%%%%%%%%%%%%%%%%%%%%%%%%%%%%%%%%%%%%%%%%%%%%%%%%
\paragraph{End of the proof of Theorem \ref{thm:GrothKatzalg}.}
%%%%%%%%%%%%%%%%%%%%%%%%%%%%%%%%%%%%%%%%%%%%%%%%%%%%%%%%%%%%%%%%%
We assume $(\cH)$. We will deduce Theorem \ref{thm:GrothKatzalg} from the analogous results in
\cite{DVInv}, where $K$ is assumed to be a number field.
To do so, we will consider the transcendence basis of $K/Q$ as a set of parameter
that we will specialize in the algebraic closure of $Q$.
We will need the following (very easy) lemma:

\begin{lemma}\label{lemma:infproduct}
Let $F$ be a field and $q$ be an element of $F$, not a root of unity.
We consider a $q$-difference system $Y(qx)=A_0(x)Y(x)$ such that $A_0(x)\in \GL_\nu(F(x))$,
zero is not a pole of $A_0(x)$ and such that $A_0(0)$ is the
identity matrix.
Then, for any norm $|~|$ (archimedean or ultrametric) over $F$ such that $|q|>1$ the formal solution
$$
Z_0(x)=\Big(A_0(q^{-1}x)A_0(q^{-2}x)A_0(q^{-3}x)\dots\Big)
$$
of $Y(qx)=A_0(x)Y(x)$ is a germ of an analytic fundamental solution at zero having infinite
radius of meromorphy.\footnote{In the sense that the entries of $Z_0(x)$ are quotient of two entire
analytic functions with respect to $|~|$.}
\end{lemma}

\begin{proof}
Since $|q|>1$ the infinite product defining $Z_0(x)$ is convergent in the neighborhood
of zero. The fact that $Z_0(x)$ is a meromorphic function with infinite radius of meromorphy
follows from the functional equation $Y(qx)=A_0(x)Y(x)$ itself.
\end{proof}

\begin{proof}[Proof of Theorem \ref{thm:GrothKatzalg}]
One side of the implication in Theorem \ref{thm:GrothKatzalg} is trivial.
So we suppose that $\Sgq^{\kappa_v}$ induces the identity on
$M_\cA\otimes_\cA\cA/(\phi_v)$ for almost all finite places $v$ of $Q$, and we prove that
$\cM_\cA$ becomes trivial over $K(x)$.
The proof is divided into steps:

\begin{proof}[Step 0. Reduction to a purely transcendental extension $K/Q$]
Let $\ul a$ be a transcendence basis of $K/Q$ and $b$ is a primitive element
of $K/Q(\ul a)$, so that  $K=Q(\ul a,b)$.
The $q$-difference field $K(x)$ can be considered as a trivial $q$-difference module
over the field $Q(\ul a)(x)$. By restriction of scalars, the module $\cM_{K(x)}$
is also a $q$-difference module over $Q(\ul a)(x)$. Since the field $K(x)$ is a trivial $q$-difference module
over $Q(\ul a)(x)$, we have:
\begin{itemize}
\item
the module $\cM_{K(x)}$ is trivial over $K(x)$ if and only if it is trivial over
$Q(\ul a)(x)$;

\item
under the present hypothesis, there exist an algebra ${\cA^\p}$ of the form
\beq\label{eq:algebraAbis}
{\cA^\p}=\cO_Q\l[\ul a,x,\frac{1}{R(x)},\frac{1}{R(qx)},....\r],\, R(x)\in\cO_Q[\ul a,x],
\eeq
and a ${\cA^\p}$-lattice $\cM_{\cA^\p}$ of $q$-difference module $\cM_{K(x)}$ over $Q(\ul a)(x)$, such that
$\cM_{\cA^\p}\otimes_{\cA^\p} Q(\ul a,x)=\cM_{K(x)}$ as a $q$-difference module over $Q(\ul a,x)$
and $\Sgq^{\kappa_v}$ induces the identity on
$\cM_{\cA^\p}\otimes_\cA\cA/(\phi_v)$, for almost all places $v$ of $Q$.
\end{itemize}
For this reason, we can actually assume that $K$ is a purely transcendental
extension of $Q$ of degree $d>0$ and that $\cA=\cA^\p$.
We fix an immersion of $Q\hookrightarrow\ol\Q$, so that
we will think to the transcendental basis $\ul a$ as a set of parameter generically varying in
$\ol\Q^d$.
\end{proof}

\begin{proof}[Step 0bis. Initial data]
Let $K=Q(\ul a)$ and $q$ be a nonzero element of $Q$, which is not a root of unity.
We are given a $q$-difference module $\cM_\cA$ over a convenient algebra $\cA$ as above, such that $K(x)$ is
the field of fraction of $\cA$
and such that
$\Sgq^{\kappa_v}$ induces the identity on $M_\cA\otimes\cO_q/(\phi_v)$, for almost all finite places $v$.
We fix a basis $\ul e$ of $\cM_\cA$, such that
$\Sgq\ul e=\ul eA^{-1}(x)$, with $A(x)\in \GL_\nu(\cA)$.
We will rather work with the associated $q$-difference system:
\beq\label{eq:sysms}
Y(qx)=A(x)Y(x).
\eeq
It follows from Proposition \ref{prop:globalnilpMS} that $\cM_{K(x)}$ is regular
singular, with no logarithmic singularities, and that its exponents are in $q^\Z$.
Enlarging a little bit the algebra $\cA$ (more precisely replacing the polynomial $R$ by a multiple of $R$),
we can suppose that both $0$ and $\infty$ are not
poles of $A(x)$ and that $A(0),A(\infty)$ are diagonal matrices with
eigenvalues in $q^\Z$ (\cf \cite[\S2.1]{Sfourier}).
\end{proof}

\begin{proof}
[Step 1. Construction of canonical solutions at $0$]
We construct a fundamental matrix of solutions, applying the Frobenius al\-go\-rithm to this particular situation
(\cf \cite{vdPutSingerDifference} or \cite[\S1.1]{Sfourier}).
There exists a shearing transformation $S_0(x)\in \GL_\nu(K[x,x^{-1}])$ such that
$$
S_0^{-1}(qx)A(x)S_0(x)=A_0(x)
$$
and $A_0(0)$ is the identity matrix. In particular, the matrix $S_0(x)$ can be written as a
product of invertible constant matrices and diagonal matrix with integral powers of $x$ on the diagonal.
Once again, up to a finitely generated extension of the algebra $\cA$, obtained inverting a convenient polynomial, we can suppose that
$S_0(x)\in \GL_\nu(\cA)$.
\par
Notice that, since $q$ is not a root of unity, there always exists
a norm, non necessarily archimedean, on $Q$ such that $|q|>1$.
We can always extend such a norm to $K$.
Then the system
\beq\label{eq:sysmsZ}
Z(qx)=A_0(x)Z(x)
\eeq
has a unique convergent solution $Z_0(x)$, as in Lemma \ref{lemma:infproduct}.
This implies that $Z_0(x)$ is a germ of a meromorphic function
with infinite radius of meromorphy.
So we have the following meromorphic solution of $Y(qx)=A(x)Y(x)$:
$$
Y_0(x)=\Big(A_0(q^{-1}x)A_0(q^{-2}x)A_0(q^{-3}x)\dots\Big)\,S_0(x).
$$
We remind that this formal infinite product represent a meromophic fundamental solution
of $Y(qx)=A(x)Y(x)$ for any norm over $K$ such that $|q|>1$ (\cf Lemma \ref{lemma:infproduct}).
\end{proof}

\begin{proof}
[Step 2. Construction of canonical solutions at $\infty$]
In exactly the same way we can construct a solution at $\infty$ of the form
$Y_\infty(x)=Z_\infty(x)S_\infty(x)$, where
the matrix $S_\infty$ belongs to $GL_\nu(K[x,x^{-1}])\cap \GL_\nu(\cA)$ and has the same form as $S_0(x)$, and
$Z_\infty(x)$ is analytic in a neighborhood of $\infty$, with $Z_\infty(\infty)=1$:
$$
Y_\infty(x)=\Big(A_\infty(x)A_\infty(qx)A_\infty(q^2x)\dots\Big)\,S_\infty(x).
$$
\end{proof}

\begin{proof}
[Step 3. The Birkhoff matrix]
To summarize we have constructed two fundamental matrices of solutions, $Y_0(x)$ at zero
and $Y_\infty(x)$ at $\infty$,
which are meromorphic over $\mathbb A^1_K\smallsetminus\{0\}$
for any norm on $K$ such that $|q|>1$, and
such that their set of poles and zeros is contained in the
$q$-orbits of the set of poles at zeros of $A(x)$.
The Birkhoff matrix
$$
B(x)=Y_0^{-1}(x)Y_\infty(x)=S_0(x)^{-1}Z_0(x)^{-1}Z_\infty(x)S_\infty(x)
$$
is a meromorphic matrix on $\mathbb A^1_K\smallsetminus\{0\}$ with elliptic entries: $B(qx)=B(x)$.
All the zeros and poles of $B(x)$, other than $0$ and $\infty$, are
contained in the $q$-orbit of zeros and poles of the matrices $A(x)$ and $A(x)^{-1}$.
\end{proof}

\begin{proof}
[Step 4.Rationality of the Birkhoff matrix]
Let us choose $\ul\a=(\a_1,\dots,\a_r)$, with $\a_i$ in the algebraic closure $\ol\Q$ of $Q$, so that we can specialize
$\ul a$ to $\ul\a$ in the coefficients of $A(x),A(x)^{-1},S_0(x),S_\infty(x)$ and that the
specialized matrices are still invertible.
Then we obtain a $q$-difference system with coefficients in $Q(\ul\a)$.
It follows from Lemma \ref{lemma:infproduct} that
for any norm on $Q(\ul\a)$ such that
$|q|>1$ we can specialize $Y_0(x),Y_\infty(x)$ and therefore $B(x)$ to matrices with meromorphic entries on $Q(\ul\a)^*$.
We will write $A^{(\ul\a)}(x)$, $Y_0^{(\ul\a)}(x)$, etc. for the specialized matrices.
\par
Since $A_{\kappa_v}(x)$ is the identity modulo $\phi_v$, the same holds for $A^{(\ul\a)}_{\kappa_v}(x)$. Therefore
the reduced system has zero $\kappa_v$-curvature modulo $\phi_v$ for almost all $v$.
We know from \cite{DVInv}, that $Y_0^{(\ul\a)}(x)$ and $Y_\infty^{(\ul\a)}(x)$ are the germs at
zero of rational functions, and therefore that $B^{(\ul\a)}(x)$ is a constant matrix
in $\GL_\nu(Q(\ul\a))$.
\par
As we have already pointed out, $B(x)$ is $q$-invariant meromorphic matrix
on $\P^1_K\smallsetminus\{0,\infty\}$. The set of its poles and zeros is the union of a finite numbers of $q$-orbits of the forms
$\be q^\Z$, such that $\be$ is algebraic over $K$ and is a pole or a zero of $A(x)$ or $A(x)^{-1}$.
If $\be$ is a pole or a zero of an entry $b(x)$ of $B(x)$ and $h_\be(x),k_\be(x)\in Q[\ul a,x]$ are the minimal polynomials of $\be$ and $\be^{-1}$
over $K$, respectively,
then we have:
$$
b(x)=\la\frac{\prod_\ga\prod_{n\geq 0}h_\ga(q^{-n}x) \prod_{n\geq 0}k_\ga(1/q^nx)}
{\prod_\de\prod_{n\geq 0}h_\de(q^{-n}x) \prod_{n\geq 0}k_\de(1/q^nx)},
$$
where $\la\in K$ and $\ga$ and $\de$ vary in a system of representatives of the $q$-orbits of the zeroes and the poles of $b(x)$, respectively.
We have proved that there exists a Zariski open set of $\ol\Q^d$ such that the specialization of $b(x)$ at any point of this set
is constant. Since the factorization written above must specialize to a convergent factorization of the same form of the corresponding
element of $B^{\ul\a}(x)$,
we conclude that $b(x)$, and therefore $B(x)$ is a constant.
\end{proof}

The fact that $B(x)\in \GL(K)$ implies that the solutions $Y_0(x)$ and $Y_\infty(x)$ glue to a meromorphic solution
on $\P^1_K$ and ends the proof of Theorem \ref{thm:GrothKatzalg}.
\end{proof}

%%%%%%%%%%%%%%%%%%%%%%%%%%%%%%%%%%%%%%%%%%%%%%%%%%%%%%%%%%%%%%%%%%%%%%%%%%%%%%%
\subsection{Generic Galois group}
\label{subsec:characomplex}
%%%%%%%%%%%%%%%%%%%%%%%%%%%%%%%%%%%%%%%%%%%%%%%%%%%%%%%%%%%%%%%%%%%%%%%%%%%%%%%

For any field $K$ of zero characteristic, any $q\in K\smallsetminus\{0,1\}$ and any $q$-difference module
$\cM_{K(x)}=(M_{K(x)},\Sgq)$ we can define
as in the previous sections the generic Galois group $Gal(\cM_{K(x)},\eta_{K(x)})$.
If $K$ is a finitely generated extension of $\Q$, in the notation of Theorem \ref{thm:car0modules}, we have:

\begin{thm}\label{thm:car0modulesgenGalois}
The generic Galois group $Gal(\cM_{K(x)},\eta_{K(x)})$ is the smallest
algebraic subgroup of $\GL(M_{K(x)})$ that contains the $v$-curvatures of the $q$-difference module $\cM_{K(x)}$ modulo $\phi_v$,
for all $v$ in a nonempty cofinite subset of $\cC$.
\end{thm}

The group $Gal(\cM_{K(x)},\eta_{K(x)})$ is a stabilizer of a line
$L_{K(x)}$ in a construction $\cW_{K(x)}=(W_{K(x)},\Sgq)$ of $\cM_{K(x)}$.
The statement above says that we can find a $\sgq$-stable algebra $\cA\subset K(x)$ of one of the forms described above,
and a $\Sgq$-stable $\cA$-lattice $M$ of $M_{K(x)}$ such that
$M$ induces an $\cA$-lattice $L$ of $L_{K(x)}$ and $W$ of $W_{K(x)}$ with the following properties:
the reduction modulo $\phi_v$ of $\Sgq^{\kappa_v}$ stabilizes $L\otimes_\cA\cA/(\phi_v)$ inside $W\otimes_\cA\cA/(\phi_v)$,
for any $v$ in a nonempty cofinite subset of $\cC$.
\par
Theorem \ref{thm:car0modulesgenGalois} has been proved in \cite[Chapter 6]{Hendrikstesi} when $q$ is a root of unity,
in the previous sections when $q$ is transcendental and in \cite{DVInv}
when $q$ is algebraic and $K$ is a number field. The remaining case
(\ie $q$ algebraic and $K$ is transcendental finitely generated over $\Q$)
is proved exactly as Theorem \ref{thm:genGalois} and
\cite[Theorem 10.2.1]{DVInv}. \par

%%%%%%%%%%%%%%%%%%%%%%%%%%%%%%%%%%%%%%%%%%%%%%%%%%%%%%%%%%%%%%%%%%%%%%%%%%%%%%%
%%%%%%%%%%%%%%%%%%%%%%%%%%%%%%%%%%%%%%%%%%%%%%%%%%%%%%%%%%%%%%%%%%%%%%%%%%%%%%%
\subsection{Generic
Galois group of a $q$-difference module over $\C(x)$, for $q\not=0,1$}
\label{subsec:complexmodules}
%%%%%%%%%%%%%%%%%%%%%%%%%%%%%%%%%%%%%%%%%%%%%%%%%%%%%%%%%%%%%%%%%%%%%%%%%%%%%%%
%%%%%%%%%%%%%%%%%%%%%%%%%%%%%%%%%%%%%%%%%%%%%%%%%%%%%%%%%%%%%%%%%%%%%%%%%%%%%%%

We deduce from the previous section a curvature characterization of the generic Galois group of a
$q$-difference module over $\C(x)$, for $q\in\C\smallsetminus\{0,1\}$.\footnote{All the statements in this subsection
remain true if one replace $\C$ with any field of characteristic zero.}

\medskip
Let $\cM_{\C(x)}=(M_{\C(x)},\Sgq)$ be a $q$-difference module over $\C(x)$.
We can consider a finitely generated extension of $K$ of $\Q$ such that there exists a
$q$-difference module $\cM_{K(x)}=(M_{K(x)},\Sgq)$ satisfying
$\cM_{\C(x)}=\cM_{K(x)}\otimes_{K(x)}\C(x)$.
First of all let us notice that:

\begin{lemma}
The $q$-difference module $\cM_{\C(x)}=(M_{\C(x)},\Sgq)$ is trivial if and only if
$\cM_{K(x)}$ is trivial.
\end{lemma}

\begin{proof}
If $\cM_{K(x)}$ is trivial, then $\cM_{\C(x)}$ is of course trivial.
The inverse statement is equivalent to the following claim. If a
linear $q$-difference system $Y(qx)=A(x)Y(x)$, with $A(x)\in \GL_\nu(K(x))$, has a
fundamental solution $Y(x)\in \GL_\nu(\C(x))$, then $Y(x)$ is actually defined over $K$.
In fact, the system $Y(qx)=A(x)Y(x)$ must be regular singular with exponents in $q^\Z$, therefore the Frobenius algorithm
allows to construct a solution $\wtilde Y(x)\in \GL_\nu(K((x)))$. We can look at $Y(x)$ as an element of $\GL_\nu(\C((x)))$.
Then there must exists a constant matrix
$C\in \GL_\nu(\C)$ such that $Y(x)=C\wtilde Y(x)$. This proves that $\wtilde Y(x)$ is the expansion of a matrix with entries in $K(x)$.
\end{proof}

With an abuse of language, Theorem \ref{thm:car0modules} can be rephrased as:

\begin{thm}\label{thm:complexmodules}
The $q$-difference module $\cM_{\C(x)}=(M_{\C(x)},\Sgq)$ is trivial if and only if
there exists a nonempty cofinite set of curvatures of $\cM_{K(x)}$, that are all zero.
\end{thm}

We can of course define as in the previous sections a generic Galois group $Gal(\cM_{K(x)},\eta_{K(x)})$.
A noetherianity argument, that we have already used several times,
shows the following:

\begin{prop}\label{prop:finitegenextension}
In the notation above we have:
$$
Gal(\cM_{\C(x)},\eta_{\C(x)})
\subset Gal(\cM_{K(x)},\eta_{K(x)})\otimes_{K(x)}\C(x).
$$
Moreover there exists a finitely generated extension $K^\p$ (resp. $K^{\p\p}$) of $K$ such that
$$
Gal(\cM_{K(x)}\otimes_{K(x)}{K^\p(x)},\eta_{K^\p(x)})\otimes_{K^\p(x)}\C(x)\cong Gal(\cM_{\C(x)},\eta_{\C(x)}).
$$
\end{prop}

Choosing $K$ large enough, we can assume that $K=K^\p$, which we will do implicitly
in the following informal statement.
We can deduce
from Theorem \ref{thm:car0modulesgenGalois}:

\begin{thm}\label{thm:complexmodulesgenGalois}
The generic Galois group $Gal(\cM_{\C(x)},\eta_{\C(x)})$ is the smallest
algebraic subgroup of $\GL(M_{\C(x)})$ that
contains a nonempty cofinite set of curvatures of the $q$-difference module $\cM_{K(x)}$.
\end{thm}

%%%%%%%%%%%%%%%%%%%%%%%%%%%%%%%%%%%%%%%%%%%%%%%%%%%%%%%%%%%%%%%%%%%
%%%%%%%%%%%%%%%%%%%%%%%%%%%%%%%%%%%%%%%%%%%%%%%%%%%%%%%%%%%%%%%%%%%
%%%%%%%%%%%%%%%%%%%%%%%%%%%%%%%%%%%%%%%%%%%%%%%%%%%%%%%%%%%%%%%%%%%
%%%%%%%%%%%%%%%%%%%%%%%%%%%%%%%%%%%%%%%%%%%%%%%%%%%%%%%%%%%%%%%%%%%
%\nocite{diviziohardouinCRAS,graniercras}
%\nocite{diviziohardouin}
%\bibliography{qG}

\newcommand{\noopsort}[1]{}

\end{document}